\documentclass[preprint,10pt]{elsarticle}

\usepackage{amsmath,amssymb,enumitem,float}
\usepackage{amsthm}
\usepackage{tikz,pgfplots,bm}

\newcommand{\supp}{\mathrm{supp\,}} 
\newcommand{\GL}{\operatorname{GL}}

\newcommand{\C}{\mathbb C}
\newcommand{\R}{\mathbb R}
\newcommand{\Q}{\mathbb Q}
\newcommand{\Z}{\mathbb{Z}}
\newcommand{\N}{\mathbb{N}}

\newtheorem{theorem}{Theorem}
\newtheorem{corollary}[theorem]{Corollary}
\newtheorem{lemma}[theorem]{Lemma}
\newtheorem{proposition}[theorem]{Proposition}
\newtheorem{remark}[theorem]{Remark}
\newtheorem*{definition}{Definition}        
\newtheorem*{example}{Example}              
\newtheorem*{openProb}{Open Problem}        

\begin{document}

\begin{frontmatter}



\title{On construction of bounded sets not admitting a general type of Riesz spectrum}


\author{Dae Gwan Lee}

\address{Mathematisch-Geographische Fakult\"at, Katholische Universit\"at Eichst\"att-Ingolstadt, \\
85071 Eichst{\"a}tt, Germany\\
Email: daegwans@gmail.com}



\begin{abstract}
Despite the recent advances in the theory of exponential Riesz bases, it is yet unknown whether there exists a set $S \subset \mathbb{R}^d$ which does not admit a Riesz spectrum, meaning that for every $\Lambda \subset \mathbb{R}^d$ the set of exponentials $e^{2\pi i \lambda \cdot x}$ with $\lambda\in\Lambda$ is not a Riesz basis for $L^2(S)$.
As a meaningful step towards finding such a set, we construct a set $S \subset [-\frac{1}{2}, \frac{1}{2}]$ which does not admit a Riesz spectrum containing a nonempty periodic set with period belonging in $\alpha \mathbb{Q}_+$ for any fixed constant $\alpha > 0$, where $\mathbb{Q}_+$ denotes the set of all positive rational numbers. In fact, we prove a slightly more general statement that the set $S$ does not admit a Riesz spectrum containing arbitrarily long arithmetic progressions with a fixed common difference belonging in $\alpha \mathbb{N}$.
Moreover, we show that given any countable family of separated sets $\Lambda_1, \Lambda_2, \ldots \subset \mathbb{R}$ with positive upper Beurling density, one can construct a set $S \subset [-\frac{1}{2}, \frac{1}{2}]$ which does not admit the sets $\Lambda_1, \Lambda_2, \ldots$ as Riesz spectrum. An interesting consequence of our results is the following statement.
There is a set $V \subset [-\frac{1}{2}, \frac{1}{2}]$ with arbitrarily small Lebesgue measure such that for any $N \in \mathbb{N}$ and any proper subset $I$ of $\{ 0, \ldots, N-1 \}$, the set of exponentials $e^{2\pi i k x}$ with $k \in \cup_{n \in I} (N\mathbb{Z} {+} n)$ is not a frame for $L^2(V)$.
The results are based on the proof technique of Olevskii and Ulanovskii in 2008.
\end{abstract}




\end{frontmatter}



\section{Introduction and Main Results}
\label{sec:intro}

One of the fundamental research topics in Fourier analysis is the theory of exponential bases and frames.
The elementary fact that $\{ e^{2 \pi i n \cdot x} \}_{n \in \Z^d}$ forms an orthogonal basis for $L^2[-\frac{1}{2}, \frac{1}{2}]^d$,
has far-reaching implications in many areas of mathematics and engineering.
For instance, the celebrated Whittaker-Shannon-Kotel'nikov sampling theorem in sampling theory is an important consequence of this fact (see e.g., \cite{Hi96}).

As a natural generalization of the functions $\{ e^{2 \pi i n \cdot x} \}_{n \in \Z^d}$ in $L^2[-\frac{1}{2}, \frac{1}{2}]^d$,
one considers the set of exponentials $E(\Lambda) := \{e^{2\pi i \lambda \cdot x} : \lambda\in\Lambda\}$,
where $\Lambda \subset \R^d$ is a discrete set consisting of the pure frequency components of exponentials (thus called the \emph{frequency set} or \emph{spectrum}), in the Hilbert space $L^2(S)$ for a finite positive measure set $S \subset \R^d$.
That is, for each $\lambda\in\Lambda$ the map $x \mapsto e^{2\pi i \lambda \cdot x}$ restricted to the set $S$ is considered as a function in $L^2(S)$.
Characterizing the properties of $E(\Lambda)$ in the space $L^2(S)$,
such as whether $E(\Lambda)$ forms an orthogonal/Riesz basis or a frame, has been an important problem in nonharmonic Fourier analysis.
The problem has a close connection to the theory of entire functions of exponential type in complex analysis,
through the celebrated work of Paley and Wiener \cite{PW34}.
For more details on this connection and for some historical background, we refer to the excellent book by Young \cite{Yo01}.
Below we give a short overview of some known results on exponential bases and frames.

\subsection{An overview of existing work on exponential bases and frames}
\label{subsec:known-results-overview}

\textbf{Exponential orthogonal bases}. \
For the case of orthogonal bases, Fuglede \cite{Fu74} posed a famous conjecture (also called the \emph{spectral set conjecture}) which states that if $S \subset \R^d$ is a finite positive measure set, then
there is an exponential orthogonal basis $E(\Lambda)$ (with $\Lambda \subset \R^d$) for $L^2(S)$
if and only if the set $S$ tiles $\R^d$ by translations along a discrete set $\Gamma \subset \R^d$ in the sense that
\begin{equation}\label{eqn:FugledeConj-S-tiles-Rd}
\sum_{\gamma \in \Gamma} \chi_S (x + \gamma) = 1
\quad \text{for a.e.} \;\; x \in \R^d ,
\end{equation}
where $\chi_S (x) = 1$ for $x \in S$, and $0$ otherwise.
The conjecture turned out to be false for $d \geq 3$ but is still open for $d = 1,2$.
There are many special cases where the conjecture is known to be true.
For instance, the conjecture is true when $\Gamma$ is a lattice of $\R^d$, in which case the set $\Lambda \subset \R^d$ can be chosen to be the dual lattice of $\Gamma$ \cite{Fu74}, and also when $S \subset \R^d$ is a convex set of finite positive measure for all $d \in \N$ \cite{LM19}.
In particular, it was shown in \cite{IKP99} that there is no exponential orthogonal basis for $L^2(S)$ when $S$ is the unit ball of $\R^d$ for $d \geq 2$, in contrast to the case $d = 1$ where the unit ball is simply $S = [-1,1]$ and $E(\frac{1}{2}\Z)$ is an orthogonal basis for $L^2 [-1,1]$.
For more details of the recent progress on Fuglede's conjecture, we refer to \cite{LM19} and the reference therein.

\medskip

\textbf{Exponential Riesz bases}. \
The relaxed case of Riesz bases is yet more challenging.
Certainly, relaxing the condition of orthogonal bases to Riesz bases allows for potentially much more feasible sets $S \subset \R^d$.
However, there are only several classes of sets $S \subset \R^d$ that are known to \emph{admit a Riesz spectrum}, meaning that there \emph{exists an exponential Riesz basis for $L^2(S)$}.
For instance, the class of convex symmetric polygons in $\R^2$ \cite{LR00}, the class of sets that are finite unions of intervals in $\R^d$ \cite{KN15,KN16}, and
the class of certain symmetric convex polytopes in $\R^d$ for all $d \geq 1$ \cite{DL19}.
On the other hand, to the best of our knowledge, nobody was able to find a set $S \subset \R^d$ which does \emph{not} admit a Riesz spectrum.

In search for an analogue of Fuglede's conjecture for Riesz bases, Grepstad and Lev \cite{GL14} considered the sets $S \subset \R^d$ that satisfy for some discrete set $\Gamma \subset \R^d$ and some $k \in \N$,
\[
\sum_{\gamma \in \Gamma} \chi_S (x + \gamma) = k
\quad \text{for a.e.} \;\; x \in \R^d .
\]
Such a set $S \subset \R^d$ is called a \emph{$k$-tile} with respect to $\Gamma$; in particular, the set $S$ satisfying \eqref{eqn:FugledeConj-S-tiles-Rd} is a $1$-tile with respect to $\Gamma$.
It was shown in \cite{GL14} that
if $S \subset \R^d$ is a bounded $k$-tile set with respect to a lattice $\Gamma \subset \R^d$ and has measure zero boundary, then the set $S$ admits a Riesz spectrum $\Lambda \subset \R^d$.
Moreover in this case, the set $\Lambda$ can be chosen to be a union of $k$ translations of $\Gamma^*$ (referred to as a $(k,\Gamma^*)$-\emph{structured spectrum}), where $\Gamma^* := (A^{-1})^T \, \mathbb{Z}^d$ is the dual lattice of $\Gamma = A \, \mathbb{Z}^d$ with $A \in \GL(d,\mathbb{R})$.
Later, Kolountzakis \cite{Ko15} gave a simpler proof of this result using elementary arguments and also eliminated the requirement of $S$ having measure zero boundary.
The converse of the statement was proved by Agora et al.~\cite{AAC15}, thus establishing the equivalence: Given a lattice $\Gamma \subset \R^d$, a bounded set $S \subset \R^d$ is a $k$-tile with respect to $\Gamma$ if and only if
it admits a $(k,\Gamma^*)$-structured Riesz spectrum.
They also showed that the \emph{boundedness} of $S$ is essential by constructing an \emph{unbounded} $2$-tile set $S \subset \R$ with respect to $\Z$ which does not admit a $(2,\Z)$-structured Riesz spectrum.
Nevertheless, for \emph{unbounded} multi-tiles $S \subset \R^d$ with respect to a lattice $\Gamma$, Cabrelli and Carbajal \cite{CC18}
were able to provide a sufficient condition for $S$ to admit a structured Riesz spectrum.
Recently, Cabrelli et al.~\cite{CHM21} found a necessary and sufficient condition for a multi-tile $S \subset \R^d$ of finite positive measure to admit a structured Riesz spectrum, which is given in terms of the Bohr compactification of the tiling lattice $\Gamma$.

\medskip

\textbf{Exponential frames}. \
Since frames allow for redundancy, it is relatively easier to obtain exponential frames than exponential Riesz bases.
For instance, the set of exponentials $\{ e^{2 \pi i n \cdot x} \}_{n \in \Z^d}$ is an orthonormal basis for $L^2[-\frac{1}{2}, \frac{1}{2}]^d$ and thus a frame for $L^2(S)$ with frame bounds $A = B =1$ whenever $S$ is a measurable subset of $[-\frac{1}{2}, \frac{1}{2}]^d$.

Nitzan et al.~\cite{NOU16} proved that
if $S \subset \R^d$ is a finite positive measure set, then there exists an exponential frame $E(\Lambda)$ (with $\Lambda \subset \R^d$) for $L^2(S)$ with frame bounds $c \, |S|$ and $C \, |S|$, where $0 < c < C < \infty$ are absolute constants.
The proof is based on a lemma from  Marcus et al.~\cite{MSS16} which resolved the famous Kadison-Singer problem in the affirmative.

\medskip

\textbf{Universality}. \
In \cite{OU06,OU08}, Olevskii and Ulanovskii considered the interesting question of universality.
They discovered some frequency sets $\Lambda \subset \R^d$ that have universal properties,
namely the so-called \emph{universal} uniqueness/sampling/
interpolation sets $\Lambda \subset \R^d$ for Paley-Wiener spaces $PW(S)$ with \emph{all} sets $S \subset \R^d$ in a certain class.
In our notation, this corresponds to the set of exponentials $E(\Lambda)$ being a complete sequence/frame/Riesz sequence in $L^2(S)$ for \emph{all} sets $S \subset \R^d$ in a certain class.
For the convenience of readers, we include a short exposition on the relevant notions in Paley-Wiener spaces in \ref{sec:PW-language}.

It was shown that universal complete sets of exponentials exist, for instance, the set $E(\Lambda)$ with
$\Lambda = \{ \ldots, -6, -4, -2 , 1, 3, 5, \ldots \}$
is complete in $L^2(S)$ for every measurable set $S \subset [-\frac{1}{2}, \frac{1}{2}]$ with $|S| \leq \frac{1}{2}$.
Furthermore,
any set $E(\Lambda)$ with $\Lambda = \{ \lambda_n \}_{n \in \Z}$ satisfying
$0 < | \lambda_n - n | \leq 1/2^{|n|}$
for all $n \in \Z$, is complete in $L^2(S)$ whenever $S \subset \R$ is a bounded measurable set with $|S| < 1$.

On the other hand, the existence of universal exponential frames and universal exponential Riesz sequences depend on the topological properties of $S$.
As a positive result,
it was shown that there is a perturbation $\Lambda$ of $\Z$ such that $E(\Lambda)$ is a frame for $L^2(S)$ whenever $S \subset \R$ is a \emph{compact} set with $|S| < 1$; a different construction of such a set $\Lambda \subset \R$ was given by Matei and Meyer \cite{MM08,MM10} based on the theory of quasicrystals.
Similarly, there is a perturbation $\Lambda$ of $\Z$ such that $E(\Lambda)$ is a Riesz sequence in $L^2(S)$ whenever $S \subset \R$ is an \emph{open} set with $|S| > 1$.
However, in the negative side, it was shown that given any $0 < \epsilon < 2$ and a separated set $\Lambda \subset \R$ with $D^-(\Lambda) < 2$, there is a measurable set $S \subset [0,2]$ with $|S| < \epsilon$ such that $E(\Lambda)$ is not a frame for $L^2(S)$, indicating that the \emph{compactness} of $S$ in the aforementioned result cannot be dropped.
Similarly, it was shown that given any $0 < \epsilon < 2$ and a separated set $\Lambda \subset \R$ with $D^+(\Lambda) > 0$, there is a measurable set $S \subset [0,2]$ with $|S| > 2 - \epsilon$ such that $E(\Lambda)$ is not a Riesz sequence in $L^2(S)$, indicating similarly that the restriction to \emph{open} sets cannot be dropped.

For more details on the universality results, we refer to Lectures 6 and 7 in the excellent lecture book by Olevskii and Ulanovskii \cite{OU16-book}.

\subsection{Contribution of the paper}
\label{subsec:contribution}

The current paper is motivated by the following open problem which was mentioned above.

\begin{openProb}
Is there a bounded/unbounded set $S \subset \R^d$ which does not admit a Riesz spectrum,
meaning that for every $\Lambda \subset \R^d$ the set of exponentials $\{e^{2\pi i \lambda \cdot x} : \lambda\in\Lambda\}$ is not a Riesz basis for $L^2(S)$?
\end{openProb}

We believe that the answer is \emph{positive}, and in this paper we make
a meaningful step towards finding such a set $S$.
Adapting the proof technique of Olevskii and Ulanovskii \cite{OU08}, we will construct a bounded subset of $\R$ which does not admit a certain general type of Riesz spectrum.
As the proof technique of \cite{OU08} works also in higher dimensions (see the end of Section 1 in \cite{OU08}), our results can be extended to higher dimensions to obtain a bounded subset of $\R^d$ with similar properties.
For simplicity of presentation, we will only consider the dimension one case ($d=1$).

Let us point out that the stated problem has been
deemed difficult by many researchers in the field, see for instance, \cite[Section 1]{CHM21}.
To resolve the problem in full may require a far more advanced proof technique than the one used in this paper.

Before presenting our results, note that for any bounded set $S \subset \R$ there are some parameters $\sigma > 0$ and $a \in \R$ such that $\frac{1}{\sigma} S + a \subset [-\frac{1}{2}, \frac{1}{2}]$.
It is therefore enough to restrict our attention to sets $S \subset [-\frac{1}{2}, \frac{1}{2}]$ (see Lemma \ref{lem:RB-basic-operations} below).
Also, recall that a set $S \subset \R$ is said to \emph{admit a Riesz spectrum} $\Lambda \subset \R$ if the set $E(\Lambda)$ is a Riesz basis for $L^2(S)$.

Our first main result is as follows.

\begin{theorem}\label{thm:main-result-S-arbit-long-AP-fixed-common-diff}
Let $0 < \alpha \leq 1$ and $0 < \epsilon < 1$.
There exists
a measurable set $S \subset [-\frac{1}{2}, \frac{1}{2}]$ with $| S | > 1 - \epsilon$ satisfying the following property:
if $\Lambda \subset \R$ contains arbitrarily long arithmetic progressions with a fixed common difference belonging in $\alpha \N$, then $E(\Lambda)$ is not a Riesz sequence in $L^2(S)$.
Moreover, such a set can be constructed explicitly as
\begin{equation}\label{eqn:first-thm-def-set-S}
S = [-\tfrac{1}{2}, \tfrac{1}{2}] \backslash V
\quad \text{with} \quad
V = [-\tfrac{1}{2} , \tfrac{1}{2}]
\cap \Big( \cup_{\ell=1}^\infty \cup_{m \in \Z} \big(\tfrac{m}{\ell\alpha} + \big[ {-} \tfrac{\epsilon}{\ell \cdot 2^{\ell+3}} , \tfrac{\epsilon}{\ell \cdot 2^{\ell+3}} \big] \big) \Big) .
\end{equation}
\end{theorem}

It should be noted that
the set $V \subset [-\frac{1}{2}, \frac{1}{2}]$ is a countable union of closed intervals, i.e., an $F_\sigma$ Borel set,
which contains $\frac{1}{\alpha} \Q \cap [-\frac{1}{2}, \frac{1}{2}]$.
Yet, the set has small Lebesgue measure $|V| < \epsilon$ due to the exponentially decreasing length of the intervals.
It is worth comparing the set $V$ with a fat Cantor set which is a closed, nowhere dense\footnote{A set is called \emph{nowhere dense} if its closure has empty interior.} subset of $[-\frac{1}{2}, \frac{1}{2}]$ with positive measure containing uncountably many elements (see e.g., \cite{Fo99,GO03}).
In contrast to the fat Cantor sets, the set $V$ is not closed and has nonempty interior. Also, the set $V$ is dense in $[-\frac{1}{2}, \frac{1}{2}]$ because it contains $\frac{1}{\alpha} \Q \cap [-\frac{1}{2}, \frac{1}{2}]$.

To illustrate the \emph{dense} set $V \subset [-\frac{1}{2}, \frac{1}{2}]$, we truncate the infinite union $\cup_{\ell=1}^\infty$ in its expression to the finite union over $\ell = 1, \ldots, 10$.
The corresponding sets for $\alpha = 1$ and $\epsilon = \frac{1}{10}, \frac{1}{2}, \frac{9}{10}$ are shown in Figure \ref{fig:SetV}.
  \begin{figure}[h!]
    \begin{center}
	\includegraphics[width=0.325\textwidth]{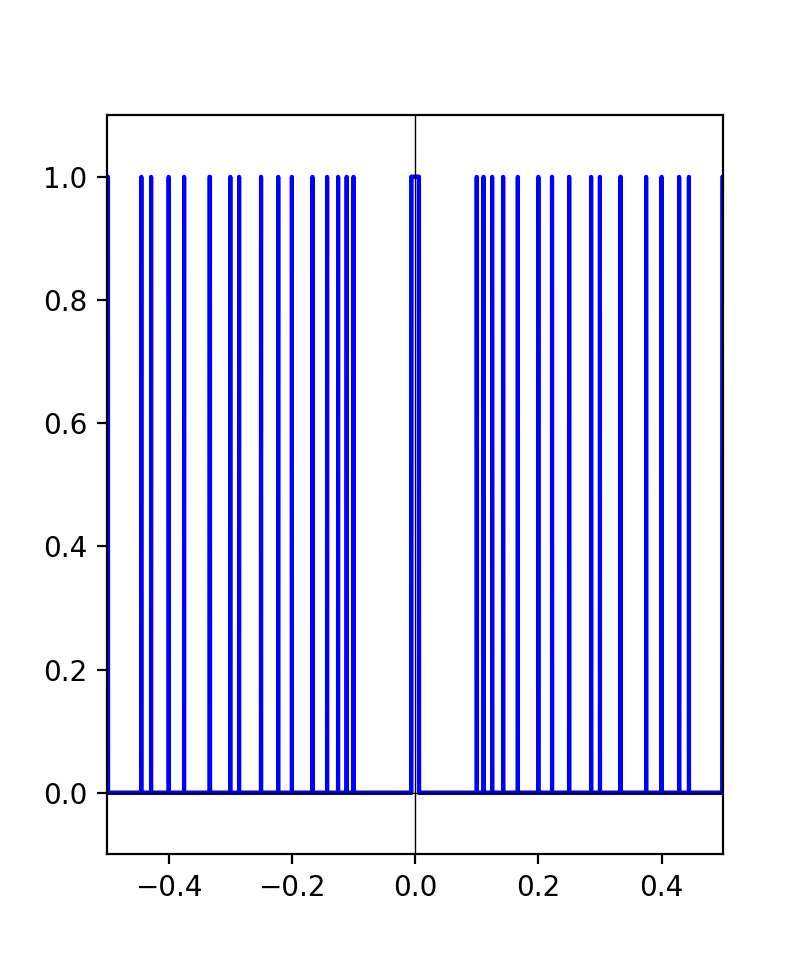}
	\includegraphics[width=0.325\textwidth]{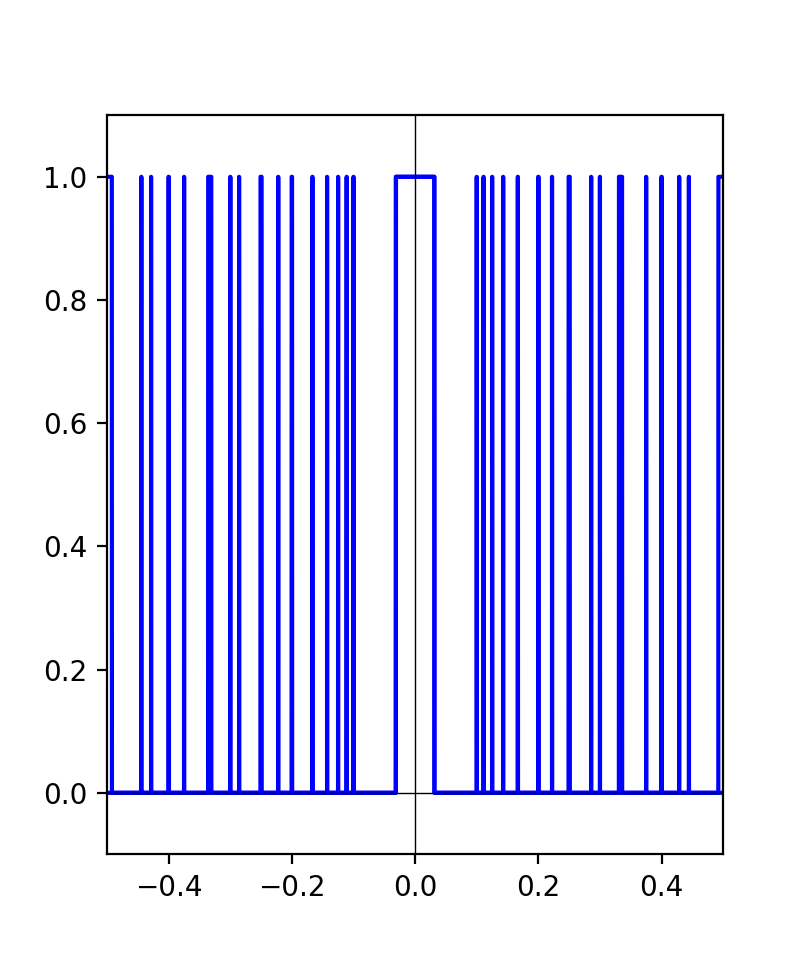}
	\includegraphics[width=0.325\textwidth]{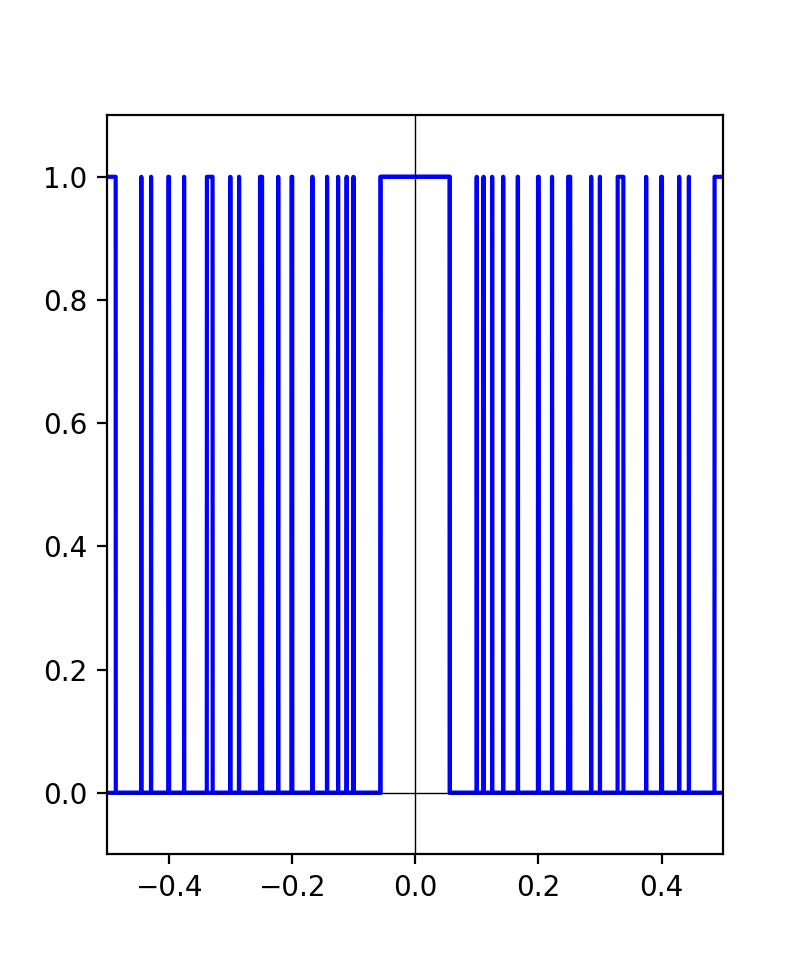}
    \end{center}
    \caption{The characteristic function of the corresponding truncated set for $\alpha = 1$ and $\epsilon = \frac{1}{10}, \frac{1}{2}, \frac{9}{10}$ (from left to right).}
    \label{fig:SetV}
  \end{figure}

To help the understanding of readers, we provide two sets $\Lambda \subset \R$, one which meets and the other which does not meet the condition stated in Theorem \ref{thm:main-result-S-arbit-long-AP-fixed-common-diff}.

\begin{example}
\rm
(a) Let $M_1 < M_2 < \cdots$ be an increasing sequence in $\N$, and let $P \in \N$.
Define the sequence $d_1 < d_2 < \cdots$ by
$d_1 = 0$ and $d_k = 2 \, \sum_{n=1}^{k-1} M_n P$ for $k \geq 2$.
Clearly, we have $d_{k+1} - d_k = 2 M_k P$ for all $k \in \N$.
Consider the set
\[
\Lambda = \pm \bigcup_{k = 1}^\infty
\big\{ d_k {+} P, \; d_k {+} 2P, \; \ldots, \; d_k {+} M_k P \big\}
\quad \subset \Z ,
\]
where $\pm \Lambda_0 := \Lambda_0 \cup (-\Lambda_0)$ for any set $\Lambda_0 \subset \R$.
This set contains arbitrarily long arithmetic progressions with common difference $P$, and has lower and upper Beurling density given by $D^-(\Lambda) = \frac{1}{2P}$ and $D^+(\Lambda) = \frac{1}{P}$, respectively
(see Section~\ref{subsec:prelim-density} for the definition of Beurling density). \\
(b) Let $N \in \N$ and let $\{ \sigma_k \}_{k=1}^\infty \subset (0,1)$ be a sequence of distinct irrational numbers between $0$ and $1$.
Consider the set
\[
\Lambda = \pm \bigcup_{k = 1}^\infty   \;
\Big( \sigma_k {+} Nk + \big\{ 0 \cdot 100^k, \; 1 \cdot 100^k, \; \ldots, \; (k{-}1) \cdot 100^k \big\} \Big)
\quad \subset \R
\]
which has uniform Beurling density $D(\Lambda) = \frac{1}{N}$.
For each $k \in \N$, the set $\Lambda$ contains exactly one arithmetic progression with common difference $100^k$ in the positive domain $(0,\infty)$, namely the arithmetic progression $\sigma_k {+} Nk$, $\sigma_k {+} Nk {+} 100^k$, $\ldots,$ $\sigma_k {+} Nk {+} (k{-}1) {\cdot} 100^k$ of length $k$.
Due to the $\pm$ mirror symmetry, the set $\Lambda$ has another such an arithmetic progression in the negative domain $(-\infty,0)$.
Note that all of these arithmetic progressions have integer-valued common difference and are distanced by some distinct irrational numbers, so none of them can be connected with another to form a longer arithmetic progression.
Hence, there is no number $P \in \N$ for which the set $\Lambda$ contains arbitrarily long arithmetic progressions with common difference $P$.
Such a set $\Lambda \subset \R$ is not covered by the class of frequency sets considered in Theorem \ref{thm:main-result-S-arbit-long-AP-fixed-common-diff}.
\end{example}

Our second main result is the following.

\begin{theorem}
\label{thm:Lambda1-Lambda2-Lambda3-introduction}
Let $0 < \epsilon < 1$ and let $\Lambda_1, \Lambda_2, \ldots \subset \R$ be a family of separated sets with $D^+ (\Lambda_{\ell}) > 0$ for all $\ell \in \N$.
One can construct a measurable set $S = S( \epsilon , \{ \Lambda_\ell \}_{\ell=1}^\infty ) \subset [-\frac{1}{2}, \frac{1}{2}]$ with $| S | > 1 - \epsilon$ such that $E(\Lambda_{\ell})$ is not a Riesz sequence in $L^2(S)$ for all $\ell \in \N$.
\end{theorem}

Let us present some interesting implications of our main results.

By convention, a discrete set $\Lambda = \{ \lambda_n \}_{n \in \Z} \subset \R$ with $\lambda_n < \lambda_{n+1}$ is called
\emph{periodic with period $t>0$} (or \emph{$t$-periodic})
if there is a number $N \in \N$ such that $\lambda_{n+N} - \lambda_n = t$ for all $n \in \Z$.
Note that if $\Lambda \subset \R$ is a nonempty periodic set with period $\alpha \cdot \frac{P}{Q} \in \alpha \Q$, where $P, Q \in \N$ are coprime numbers, then it must contain a translated copy of $\alpha P \Z$, that is,
$\alpha P \Z {+} d  \subset  \Lambda$
for some $d \in \R$.
As a result,
we have the following corollary of Theorem \ref{thm:main-result-S-arbit-long-AP-fixed-common-diff}.

\begin{corollary}\label{cor:main-result-S-without-rationally-periodic-Riesz-spectrum}
For any $0 < \alpha \leq 1$ and $0 < \epsilon < 1$,
let $S \subset [-\frac{1}{2}, \frac{1}{2}]$ be the set given by \eqref{eqn:first-thm-def-set-S}.
Then for any nonempty periodic set $\Lambda \subset \R$ with period belonging in $\alpha \Q_+ \,{=}\, \alpha \Q \cap (0,\infty)$, the set $E(\Lambda)$ is not a Riesz sequence in $L^2(S)$.
Consequently, the set $S$ does not admit a Riesz spectrum containing a nonempty periodic set with period belonging in $\alpha \Q_+$
\end{corollary}

It is worth noting that the class of nonempty periodic sets with rational period is \emph{uncountable},
because of the flexibility in placement of elements in each period;
hence, Corollary \ref{cor:main-result-S-without-rationally-periodic-Riesz-spectrum} cannot be deduced from Theorem \ref{thm:Lambda1-Lambda2-Lambda3-introduction}.

As mentioned in Section~\ref{subsec:known-results-overview}, Agora et al.~\cite{AAC15} constructed an \emph{unbounded} $2$-tile set $S \subset \R$ with respect to $\Z$ which does not admit a Riesz spectrum of the form $(\Z {+} \sigma_1 ) \cup (\Z {+} \sigma_2)$ with $\sigma_1, \sigma_2 \in \R$.
By a dilation, one could easily generalize this example to an \emph{unbounded} $2$-tile set $W \subset \R$ with respect to $\frac{1}{\alpha} \Z$ for any fixed $\alpha > 0$, which does not admit a Riesz spectrum of the form $(\alpha \Z {+} \sigma_1 ) \cup (\alpha \Z {+} \sigma_2)$ with $\sigma_1, \sigma_2 \in \R$.
Note that such a form of Riesz spectrum is $\alpha$-periodic and thus not admitted by our set $S$ given by \eqref{eqn:first-thm-def-set-S} with any $0 < \epsilon < 1$. In fact, our set $S$ has a much stronger property than $W$, namely that $S$ does not admit a periodic Riesz spectrum with period belonging in $\alpha \Q_+$ and moreover, the set $S$ is \emph{bounded}.

Since the set $S$ is contained in $[-\frac{1}{2}, \frac{1}{2}]$, it is particularly interesting to consider the frequency sets consisting of integers $\Omega \subset \Z$.
Noting that a periodic subset of $\Z$ is necessarily $N$-periodic for some $N \in \N$, we immediately deduce the following result from Corollary \ref{cor:main-result-S-without-rationally-periodic-Riesz-spectrum}.

\begin{corollary}
\label{cor:periodic-integer-set-Lambda}
Let $S \subset [-\frac{1}{2}, \frac{1}{2}]$ be the set given by \eqref{eqn:first-thm-def-set-S} with $\alpha = 1$ and any $0 < \epsilon < 1$.
Then for any nonempty periodic set $\Omega \subset \Z$, the set $E(\Omega)$ is not a Riesz sequence in $L^2(S)$.
\end{corollary}

Alternatively, one could construct such a set $S \subset [-\frac{1}{2}, \frac{1}{2}]$ from Theorem \ref{thm:Lambda1-Lambda2-Lambda3-introduction} by observing that the family of all nonempty periodic integer sets is \emph{countable}; indeed, the one and only nonempty $1$-periodic integer set is $\Z$, the nonempty $2$-periodic integer sets are $2\Z$, $2\Z{+}1$, $\Z$, and so on.

Further, it is easy to deduce the following result from Corollary \ref{cor:periodic-integer-set-Lambda} and Proposition \ref{prop:Prop5-4-BCMS19} below, by setting $V := [-\frac{1}{2}, \frac{1}{2}] \backslash S$ and $\Omega' := \Z \backslash \Omega$.

\begin{corollary}
\label{cor:periodic-integer-set-Lambda-complement-sets}
Let $S \subset [-\frac{1}{2}, \frac{1}{2}]$ be the set given by \eqref{eqn:first-thm-def-set-S} with $\alpha = 1$ and any $0 < \epsilon < 1$,
and let $V := [-\frac{1}{2}, \frac{1}{2}] \backslash S$. Then for any proper periodic subset $\Omega' \subsetneq \Z$, the set $E(\Omega')$ is not a frame for $L^2(V)$.
\end{corollary}

The significance of Corollary \ref{cor:periodic-integer-set-Lambda-complement-sets} is
in the fact that for any $N \in \N$ and any proper subset $I \subsetneq \{ 0, \ldots, N-1 \}$,
the set of exponentials $E \big( \cup_{n \in I} (N\Z {+} n) \big)$
is not a frame for $L^2(V)$ even though the set $V$ has very small Lebesgue measure $|V| < \epsilon$.
Note that $E(\Z)$ is a frame for $L^2(V)$ with frame bounds $A = B =1$, since it is an orthonormal basis for $L^2[0,1]$.

\section{Preliminaries}
\label{sec:prelim}

\subsection{Sequences in separable Hilbert spaces}

\begin{definition}\label{def:HilbertSpace-Sequences}
A sequence $\{ f_n \}_{n\in\Z}$ in a separable Hilbert space $\mathcal{H}$ is called
\begin{itemize}
\item a \emph{Bessel sequence} in $\mathcal{H}$ (with a Bessel bound $B$) if there is a constant $B>0$ such that
\[
\sum_{n\in\Z} \big| \langle f , f_n \rangle \big|^{2}
\;\leq\;
B \, \| f \|^2
\quad \text{for all} \;\; f \in \mathcal{H};
\]

\item  a \emph{frame} for $\mathcal{H}$ (with frame bounds $A$ and $B$) if there are constants $0 < A \leq B < \infty$ such that
\[
A \, \| f \|^2
\;\leq\;
\sum_{n\in\Z} |\langle f , f_n\rangle|^{2}
\;\leq\;
B \, \| f \|^2
\quad \text{for all} \;\; f \in \mathcal{H};
\]

\item  a \emph{Riesz sequence} in $\mathcal{H}$ (with Riesz bounds $A$ and $B$) if there are constants $0 < A \leq B < \infty$ such that
\[
A \, \| c \|_{\ell_2}^2
\;\leq\;
\Big\| \sum_{n\in\Z} c_n \, f_n \Big\|^2
\;\leq\;
B \, \| c \|_{\ell_2}^2
\quad \text{for all} \;\; c=\{c_n\}_{n\in\Z} \in \ell_2 (\mathbb Z);
\]

\item  a \emph{Riesz basis} for $\mathcal{H}$ (with Riesz bounds $A$ and $B$) if it is a complete Riesz sequence in $\mathcal{H}$ (with Riesz bounds $A$ and $B$);

\item  an \emph{orthogonal basis} for $\mathcal{H}$ if it is a complete sequence of nonzero elements in $\mathcal{H}$ such that $\langle f_m, f_n \rangle = 0$ whenever $m \neq n$.

\item  an \emph{orthonormal basis} for $\mathcal{H}$ if it is complete and $\langle f_m, f_n \rangle = \delta_{m,n}$ whenever $m \neq n$.
\end{itemize}
The associated bounds $A$ and $B$ are said to be \emph{optimal} if they are the tightest constants satisfying the respective inequality.
\end{definition}

In general, an orthonormal basis is a Riesz basis with Riesz bounds $A=B=1$, but an orthogonal basis is not necessarily norm-bounded below and thus generally not a Riesz basis (for instance, consider the sequence $\{ \frac{e_n}{n} \}_{n=1}^\infty$ where $\{ e_n \}_{n=1}^\infty$ is an orthonormal basis for $\mathcal{H}$).
Nevertheless, exponential functions have constant norm in $L^2(S)$ for any finite measure set $S \subset \R^d$, namely $\| e^{2 \pi i \lambda \cdot (\cdot)} \|_{L^2(S)} = \sqrt{|S|}$ for all $\lambda \in \R^d$.
Thus, an \emph{exponential orthogonal basis} is simply an \emph{exponential orthonormal basis} scaled by a common multiplicative factor.

\begin{proposition}\label{prop:facts-RieszBasesFR}
Let $\mathcal{H}$ be a separable Hilbert space.\\
(a) \cite[Corollary 3.7.2]{Ch16} Every subfamily of a Riesz basis is a Riesz sequence with the same bounds (the optimal bounds may be tighter). \\
(b) \cite[Corollary 8.24]{He11} If $\{ f_n \}_{n\in\Z}$ is a Bessel sequence in $\mathcal{H}$ with Bessel bound $B$, then $\| f_i \|^2 \leq B$ for all $i \in I$.
If $\{ f_n \}_{n\in\Z}$ is a Riesz sequence in $\mathcal{H}$ with bounds $0 < A \leq B < \infty$, then $A \leq \| f_i \|^2 \leq B$ for all $i \in I$. \\
(c) \cite[Lemma 3.6.9, Theorems 3.6.6, 5.4.1, and 7.1.1]{Ch16} (or see \cite[Theorems 7.13, 8.27, and 8.32]{He11})
Let $\{ e_n \}_{n\in\Z}$ be an orthonormal basis for $\mathcal{H}$ and let $\{ f_n \}_{n\in\Z} \subset \mathcal{H}$. The following are equivalent.
\begin{itemize}
\item
$\{ f_n \}_{n\in\Z}$ is a Riesz basis for $\mathcal{H}$.
\item
$\{ f_n \}_{n\in\Z}$ is an exact frame (i.e., a frame that ceases to be a frame whenever a single element is removed) for $\mathcal{H}$.
\item
$\{ f_n \}_{n\in\Z}$ is an unconditional basis of $\mathcal{H}$ with $0 < \inf_{n\in\Z} \| f_n \| \leq \sup_{n\in\Z} \| f_n \| < \infty$.
\item
There is a bijective bounded operator $T : \mathcal{H} \rightarrow \mathcal{H}$ such that $Te_n=f_n$ for all $n \in \Z$.
\end{itemize}
Moreover in this case, the optimal frame bounds coincide with the optimal Riesz bounds.
\end{proposition}

\begin{proposition}[Proposition 5.4 in \cite{BCMS19}]
\label{prop:Prop5-4-BCMS19}
Let $\{ e_n \}_{n \in \Z}$ be an orthonormal basis of a separable Hilbert space $\mathcal{H}$.
Let $P : \mathcal{H} \rightarrow \mathcal{M}$ be the orthogonal projection from $\mathcal{H}$ onto a closed subspace $\mathcal{M}$.
Let $J \subset \Z$ and $0 < \alpha < 1$. The following are equivalent.
\begin{itemize}
\item[$\mathrm{(\romannumeral 1)}$] $\{ P e_n \}_{n \in J} \subset \mathcal{M}$ is a Bessel sequence with optimal bound $1-\alpha$. (Note that $\{ P e_n \}_{n \in J}$ is always a Bessel sequence with bound $1$.)

\item[$\mathrm{(\romannumeral 2)}$] $\{ P e_n \}_{n \in J^c} \subset \mathcal{M}$ is a frame for $\mathcal{M}$ with optimal lower bound $\alpha$ and upper bound $1$ (not necessarily optimal).

\item[$\mathrm{(\romannumeral 3)}$] $\{ (\mathrm{Id}-P) e_n \}_{n \in J} \subset \mathcal{M}^{\perp}$ is a Riesz sequence with optimal lower bound $\alpha$ and upper bound $1$ (not necessarily optimal).
\end{itemize}
\end{proposition}

\subsection{Exponential systems}

As already introduced in Section~\ref{sec:intro}, we define the exponential system $E(\Lambda)=\{e^{2\pi i \lambda \cdot (\cdot)} : \lambda\in\Lambda\}$
for a discrete set $\Lambda \subset \R^d$ (called a \emph{frequency set} or a \emph{spectrum}).

\begin{lemma}
\label{lem:RB-basic-operations}
Assume that $E(\Lambda)$ is a Riesz basis for $L^2(S)$ with optimal bounds $0 < A \leq B < \infty$, where $\Lambda \subset \R^d$ is a discrete set and $S \subset \R^d$ is a measurable set. \\
(a) For any $a \in \R^d$, $E(\Lambda)$ is a Riesz basis for $L^2(S+a)$ with bounds $A$ and $B$. \\
(b) For any $b \in \R^d$, $E(\Lambda + b)$ is a Riesz basis for $L^2(S)$ with bounds $A$ and $B$. \\
(c) For any $\sigma > 0$, $\sqrt{\sigma} \, E(\sigma \Lambda)$ is a Riesz basis for $L^2(\frac{1}{\sigma} S)$ with bounds $A$ and $B$, equivalently, $E(\sigma \Lambda)$ is a Riesz basis for $L^2(\frac{1}{\sigma} S)$ with bounds $\frac{A}{\sigma}$ and $\frac{B}{\sigma}$.
\end{lemma}

A proof of Lemma \ref{lem:RB-basic-operations} is given in \ref{sec:proof-of-preliminaries}.

\begin{remark}
\rm
Lemma \ref{lem:RB-basic-operations} remains valid if the term ``Riesz basis'' is replaced with one of the following: ``Riesz sequence'', ``frame'', and ``frame sequence'' (and also ``Bessel sequence'' in which case the lower bound is simply neglected).
\end{remark}

\begin{theorem}[The Paley-Wiener stability theorem \cite{PW34}]
\label{thm:PaleyWienerStability}
Let $V \subset \R$ be a bounded set of positive measure and $\Lambda = \{ \lambda_n \}_{n \in \Z} \subset \R$ be a sequence of real numbers such that $E(\Lambda)$ is a Riesz basis for $L^2(V)$ (resp.~a frame for $L^2(V)$, a Riesz sequence in $L^2(V)$). There exists a constant $\theta = \theta(\Lambda,V) > 0$ such that whenever $\Lambda' = \{ \lambda_n' \}_{n \in \Z} \subset \R$ satisfies
\[
| \lambda_n' - \lambda_n | \leq \theta, \quad n \in \Z ,
\]
the set of exponentials $E(\Lambda')$ is a Riesz basis for $L^2(V)$ (resp.~a frame for $L^2(V)$, a Riesz sequence in $L^2(V)$).
\end{theorem}

For a proof of Theorem \ref{thm:PaleyWienerStability}, we refer to \cite[p.\,160]{Yo01} for the case where $V$ is a single interval, and \cite[Section 2.3]{KN15} for the general case.
It is worth noting that the constant $\theta = \theta(\Lambda,V)$ depends on the Riesz bounds of the Riesz basis $E(\Lambda)$ for $L^2(V)$, which are determined once $\Lambda$ and $V$ are given.
Also, it is pointed out in \cite[Section 2.3, Remark 2]{KN15} that the theorem holds also for frames and Riesz sequences.

\subsection{Density of frequency sets}
\label{subsec:prelim-density}

The lower and upper (Beurling) density of a discrete set $\Lambda \subset \R^d$ is defined respectively by (see e.g., \cite{He07})
\[
\begin{split}
& D^-(\Lambda) = \liminf_{r \rightarrow \infty} \frac{\inf_{x \in \R^d} |\Lambda \cap [x,x+r] |}{r}
\qquad \text{and} \\
& D^+(\Lambda) = \limsup_{r \rightarrow \infty} \frac{\sup_{x \in \R^d} |\Lambda \cap [x,x+r] |}{r} .
\end{split}
\]
If $D^-(\Lambda) = D^+(\Lambda)$, we say that $\Lambda$ has \emph{uniform} (Beurling) density $D(\Lambda) := D^-(\Lambda) = D^+(\Lambda)$.
A discrete set $\Lambda \subset \R^d$ is called \emph{separated} (or \emph{uniformly discrete}) if its separation constant $\Delta(\Lambda) := \inf\{ |\lambda - \lambda'| : \lambda \neq \lambda' \in \Lambda \}$ is positive.
We will always label the elements of a separated set $\Lambda$ in the increasing order, that is, $\Lambda = \{ \lambda_n \}_{n \in \Z}$ with $\lambda_n < \lambda_{n+1}$ for all $n \in \Z$.

The following proposition is considered folklore. The corresponding statements for Gabor systems of $L^2(\R^d)$ are well-known (see \cite[Theorem 1.1]{CDH99} and also \cite[Lemma 2.2]{GOR15}) and the following proposition can be proved similarly.

\begin{proposition}\label{prop:FR-rel-sep-RS-sep}
Let $\Lambda \subset \R^d$ be a discrete set and let $S \subset \R^d$ be a finite positive measure set which is not necessarily bounded. \\
(i) If $E(\Lambda)$ is a Bessel sequence in $L^2(S)$, then $D^+(\Lambda) < \infty$. \\
(ii) If $E(\Lambda)$ is a Riesz sequence in $L^2(S)$, then $\Lambda$ is separated, i.e., $\Delta(\Lambda) > 0$.
\end{proposition}

A proof of Proposition \ref{prop:FR-rel-sep-RS-sep} is given in \ref{sec:proof-of-preliminaries}.

\begin{theorem}[\cite{La67,NO12}]
\label{thm:Landau}
Let $\Lambda \subset \R^d$ be a discrete set and let $S \subset \R^d$ be a finite positive measure set. \\
(i) If $E(\Lambda)$ is a frame for $L^2(S)$, then $|S| \leq D^-(\Lambda) \leq D^+(\Lambda) < \infty$. \\
(ii) If $E(\Lambda)$ is a Riesz sequence in $L^2(S)$, then $\Lambda$ is separated and $D^+(\Lambda) \leq |S|$.
\end{theorem}

\begin{corollary}
\label{cor:Landau}
Let $\Lambda \subset \R^d$ be a discrete set and let $S \subset \R^d$ be a finite positive measure set.
If $E(\Lambda)$ is a Riesz basis for $L^2(S)$, then $\Lambda$ is separated and has uniform Beurling density $D(\Lambda) = |S|$.
\end{corollary}

\section{A result of Olevskii and Ulanovskii}
\label{sec:review-of-OU08}

As our main results (Theorems \ref{thm:main-result-S-arbit-long-AP-fixed-common-diff} and \ref{thm:Lambda1-Lambda2-Lambda3-introduction}) hinge on the proof technique of Olevskii and Ulanovskii \cite{OU08}, we will briefly review the relevant result in \cite{OU08}.

\begin{theorem}[Theorem 4 in \cite{OU08}]
\label{thm:Theorem4-in-OU08}
Let $0 < \epsilon < 1$ and let $\Lambda \subset \R$ be a separated set with $D^+ (\Lambda) > 0$.
One can construct a measurable set $S = S(\epsilon,\Lambda) \subset [-\frac{1}{2}, \frac{1}{2}]$ with $| S | > 1 - \epsilon$ such that $E(\Lambda)$ is not a Riesz sequence in $L^2(S)$.
\end{theorem}

The proof of Theorem \ref{thm:Theorem4-in-OU08} relies on a technical lemma (Lemma \ref{lem:Lemma5-1-in-OU08} below) which is based on the celebrated Szemer\'{e}di's theorem \cite{Sz75}
asserting that any integer set $\Omega \subset \Z$ with positive upper Beurling density\footnote{When restricted to subsets of integers $\Lambda \subset \Z$, the upper Beurling density is equal to the so-called \emph{upper Banach density} which is defined as
$\limsup_{r \rightarrow \infty} \sup_{n \in \Z} | \Lambda \cap \{ n{+}1, n{+}2\ldots, n{+}r \} | / r$.
In the literature, Szemer\'{e}di's theorem is often stated for sets $\Lambda \subset \N$ with positive \emph{upper natural density (upper asymptotic density)}
$\limsup_{r \rightarrow \infty} | \Lambda \cap \{ 1, 2, \ldots, r \} | / r > 0 $.
It should be noted that the statements of Szemer\'{e}di's theorem with different types of density are equivalent, but the proofs are not easily converted from one density type to the other.}
$D^+ (\Omega) > 0$ contains at least one arithmetic progression of length $M$ for all $M \in \N$. Here, an arithmetic progression of length $M$ means a sequence of the form
\[
d, \; d{+}P , \; d{+}2P, \; \ldots, \; d{+}(M{-}1)P
\qquad \text{with} \;\; d \in \Z \;\; \text{and} \;\; P \in \N .
\]

As a side remark, we mention that the common difference $P \in \N$ of the arithmetic progression resulting from Szemer\'{e}di's theorem,
can be restricted to a fairly sparse subset of positive integers $\mathcal{C} \subset \N$.
For instance, one can ensure that $P$ is a multiple of any prescribed number $L \in \N$, by passing to a subset of $\Omega$ that is contained in $L \Z {+} u$ for some $u \in \{ 0, 1, \ldots, L-1 \}$
and has positive upper Beurling density.
This allows us to take $\mathcal{C} = L \N$ which clearly satisfies $D^+ (\mathcal{C}) = 1/L$.
Further, one can even choose
$\mathcal{C} = \{ 1^q, 2^q, 3^q, \ldots \}$ for any $q \in \N$, which satisfies
\[
D^+ (\mathcal{C})
=
\begin{cases}
1 & \text{if} \;\; q=1 ,  \\
0 & \text{if} \;\; q>1 .
\end{cases}
\]
More generally, one may choose $\mathcal{C} = \{ p(n) : n \in \N \}$ for any polynomial $p$ with rational coefficients such that $p(0) = 0$ and $p(n) \in \Z$ for $n \in \Z \backslash \{ 0 \}$ (see \cite[p.733]{BL96}).
On the other hand, it was shown in \cite[Theorem 7]{ES77} that
$\mathcal{C} \subset \N$ cannot be a lacunary sequence, i.e., a sequence $\{ a_n \}_{n=1}^\infty$ satisfying $\liminf_{n \rightarrow \infty} a_{n+1} / a_n > 1$ (for instance, $\{ 2^n : n = 0, 1, 2, \ldots \}$).
Note that the aforementioned set $\mathcal{C} = \{ p(n) : n \in \N \}$ can be sparse but not lacunary since $\lim_{n \rightarrow \infty} p(n+1) / p(n) = 1$ for any polynomial $p$.
We refer to \cite[Section 2]{FLW16} for a short review on the possible choice of (deterministic) sets $\mathcal{C} \subset \N$, and also for the situation where $\mathcal{C}$ is chosen randomly.

\begin{lemma}[Lemma 5.1 in \cite{OU08}]
\label{lem:Lemma5-1-in-OU08}
Let $\Lambda \subset \R$ be a separated set with $D^+ (\Lambda) > 0$.
For any $M \in \N$ and $\delta > 0$, there exist constants $c = c (M, \delta, \Lambda) \in \N$, $d = d (M, \delta, \Lambda) \in \R$, and an increasing sequence $s(-M) < s(-M{+}1) < \ldots < s(M)$ in $\Lambda$ such that
\begin{equation}\label{eqn:Lemma5-1-in-OU08-condition}
\big| s(j) - c j - d \big| \leq \delta
\quad \text{for} \;\; j = -M, \ldots, M .
\end{equation}
Moreover, the constant $c = c (M, \delta, \Lambda) \in \N$ can be chosen to be a multiple of any prescribed number $L \in \N$.
\end{lemma}

As Lemma \ref{lem:Lemma5-1-in-OU08} will be used in the proof of Theorem \ref{thm:Lambda1-Lambda2-Lambda3-introduction}, we include a short proof of Lemma \ref{lem:Lemma5-1-in-OU08} in \ref{sec:proof-of-preliminaries} for self-containedness of the paper.

\section{Proof of Theorem \ref{thm:main-result-S-arbit-long-AP-fixed-common-diff}}
\label{sec:proof-first-result}

Before proving Theorem \ref{thm:main-result-S-arbit-long-AP-fixed-common-diff}, we note that Theorem \ref{thm:Lambda1-Lambda2-Lambda3-introduction} is an extension of Theorem \ref{thm:Theorem4-in-OU08} from a single set $\Lambda \subset \R$ to a countable family of sets $\Lambda_1, \Lambda_2, \ldots \subset \R$.
We will first consider a particular choice of sets
$\Lambda_1 \,{=}\, \alpha \Z$, $\Lambda_2 \,{=}\, 2 \alpha \Z$, $\Lambda_3 \,{=}\, 3 \alpha \Z, \cdots$ for any fixed $0 < \alpha \leq 1$, from which a desired set for Theorem \ref{thm:main-result-S-arbit-long-AP-fixed-common-diff} will be acquired.

\begin{proposition}\label{prop:seq-Z-2Z-3Z}
Let $0 < \alpha \leq 1$ and
$\Lambda_1 \,{=}\, \alpha \Z$, $\Lambda_2 \,{=}\, 2 \alpha \Z$, $\Lambda_3 \,{=}\, 3 \alpha \Z, \cdots$,
that is, $\Lambda_\ell \,{=}\, \ell \alpha \Z$ for $\ell \in \N$.
Given any $0 < \epsilon < 1$, one can construct a measurable set
$S \subset [-\frac{1}{2}, \frac{1}{2}]$
with $| S | > 1 - \epsilon$ such that $E(\Lambda_{\ell})$ is not a Riesz sequence in $L^2(S)$ for all $\ell \in \N$.
\end{proposition}

\begin{proof}
Fix any $0 < \epsilon < 1$ and choose an integer $R > \frac{1}{1-\epsilon}$ so that $0 < \epsilon < \frac{R-1}{R}$.
We claim that for each $0 < \eta < \frac{R-1}{R}$ there exists a set $V_{\eta} \subset [-\frac{1}{2}, \frac{1}{2}]$ with $| V_{\eta} | < \eta$ satisfying the following property: for each $\ell \in \N$ there is a finitely supported sequence $b^{(\eta,\ell)} = \{ b^{(\eta,\ell)}_j \}_{j \in \Z}$ satisfying
\begin{equation}\label{eqn:our-claim-Z-2Z-3Z}
\int_{[-\frac{1}{2}, \frac{1}{2}] \backslash V_{\eta}}
\Big| \sum_{j \in \Z} b^{(\eta,\ell)}_j \, e^{2 \pi i \ell \alpha j x} \Big|^2 \, dx
\;\leq\;
R \, \tfrac{\eta}{2^\ell}
\sum_{j \in \Z} \big| b^{(\eta,\ell)}_j \big|^2  .
\end{equation}
If this claim is proved, one could take $V := \cup_{k=1}^\infty V_{\epsilon / 2^k}$ and $S := [-\frac{1}{2}, \frac{1}{2}] \backslash V$.
Indeed, we have $|V| \leq \sum_{k=1}^\infty | V_{\epsilon / 2^k} | < \sum_{k=1}^\infty \frac{\epsilon}{2^k} = \epsilon$, so that $| S | > 1 - \epsilon$.
Also, it holds for any $k, \ell \in \N$,
\[
\begin{split}
\int_S  \Big| \sum_{j \in \Z} b^{(\epsilon/2^k,\ell)}_j \, e^{2 \pi i \ell \alpha j x} \Big|^2 \, dx
&\;\leq\;
\int_{[-\frac{1}{2}, \frac{1}{2}] \backslash V_{\epsilon / 2^k}}
\Big| \sum_{j \in \Z} b^{(\epsilon/2^k,\ell)}_j \, e^{2 \pi i \ell \alpha j x} \Big|^2 \, dx \\
&\;\overset{\eqref{eqn:our-claim-Z-2Z-3Z}}{\leq} \; R \, \tfrac{\epsilon}{2^{k+\ell}}  \sum_{j \in \Z} \big| b^{(\epsilon/2^k,\ell)}_j \big|^2 .
\end{split}
\]
By fixing any $\ell \in \N$ and letting $k \rightarrow \infty$, we conclude that $E(\ell \alpha \Z)$ is not a Riesz sequence in $L^2(S)$.

To prove the claim \eqref{eqn:our-claim-Z-2Z-3Z}, fix any $0 < \eta < \frac{R-1}{R}$.
For each $\ell \in \N$, let $\widetilde{a}^{(\eta\alpha/2^\ell)} = \{ \widetilde{a}^{(\eta\alpha/2^\ell)}_j \}_{j \in \Z} \in \ell_2(\Z)$ be the sequence given by
\begin{equation}\label{eqn:def-widetilde-a}
\widetilde{a}^{(\eta\alpha/2^\ell)}_j
:=
\begin{cases}
\sqrt{\tfrac{\eta\alpha}{2^{\ell+1}}} & \text{if} \;\; j=0 , \\[.5em]
\sqrt{\tfrac{2^{\ell+1}}{\eta\alpha}} \tfrac{1}{\pi j} \sin \big( \tfrac{\pi j \eta\alpha}{2^{\ell+1}} \big) & \text{if} \;\; j \neq 0 ,
\end{cases}
\end{equation}
which is the Fourier coefficient of the $1$-periodic function
\begin{equation}\label{eqn:widetilde-p-epsilonL-time-limited}
\widetilde{p}_{\eta\alpha/2^\ell}(x)
:=
\begin{cases}
\sqrt{\tfrac{2^{\ell+1}}{\eta\alpha}} & \text{for} \;\; x \in \big[ {-} \tfrac{\eta\alpha}{4 \cdot 2^\ell} , \tfrac{\eta\alpha}{4 \cdot 2^\ell} \big] , \\[.5em]
0 & \text{for} \;\; x \in \big[ {-} \tfrac{1}{2} , \tfrac{1}{2} \big) \big\backslash \big[ {-} \tfrac{\eta\alpha}{4 \cdot 2^\ell} , \tfrac{\eta\alpha}{4 \cdot 2^\ell} \big] ,
\end{cases}
\end{equation}
that is,
$\widetilde{p}_{\eta\alpha/2^\ell}(x) = \sum_{j \in \Z} \widetilde{a}^{(\eta\alpha/2^\ell)}_j \, e^{2 \pi i j x}$ for a.e.~$x \in [-\frac{1}{2}, \frac{1}{2}]$.
Note that
$\| \widetilde{a}^{(\eta\alpha/2^\ell)} \|_{\ell_2}$
$= \| \widetilde{p}_{\eta\alpha/2^\ell}(x) \|_{L^2[-\frac{1}{2},\frac{1}{2}]} = 1$.
Choose a number $\widetilde{M} = \widetilde{M}(\eta\alpha/2^\ell) \in \N$ satisfying
\[
\sum_{|j| > \widetilde{M}} \big| \widetilde{a}^{(\eta\alpha/2^\ell)}_j \big|^2
\;<\; \tfrac{1}{\alpha} \cdot \big( \tfrac{\eta\alpha}{2^\ell} \big)
\;=\; \tfrac{\eta}{2^\ell},
\]
so that
\begin{equation}\label{eqn:a-epsilonL-truncated-part}
\sum_{j=-\widetilde{M}}^{\widetilde{M}} \big| \widetilde{a}^{(\eta\alpha/2^\ell)}_j \big|^2
\;>\; 1 - \tfrac{\eta}{2^\ell}
\;\geq\; 1 - \eta
\;>\; \tfrac{1}{R} .
\end{equation}
Now, the set $\Lambda_\ell$ comes into play.
We write $\Lambda_\ell = \ell \alpha \Z = \{ s_{\ell} (j) : j \in \Z \}$ with
\[
s_{\ell} (j) := \ell \alpha j
\quad \text{for all} \;\; j \in \Z .
\]
For $x \in [-\frac{1}{2}, \frac{1}{2}]$, we define
\begin{equation}\label{eqn:def-f-epsilon-ell-Z-2Z-3Z}
\widetilde{f}_{\eta\alpha/2^\ell,\Lambda_\ell} (x)
:= \sum_{j=-\widetilde{M}}^{\widetilde{M}} \widetilde{a}^{(\eta\alpha/2^\ell)}_j \, e^{2 \pi i s_{\ell} (j) x}
\end{equation}
and observe that
\begin{equation}\label{eqn:f-eps-ell-minus-p-eps-ellx-UpperBound-by-eps-Z-2Z-3Z}
\widetilde{f}_{\eta\alpha/2^\ell,\Lambda_\ell} (x) - \widetilde{p}_{\eta\alpha/2^\ell} (\ell \alpha x)
= - \sum_{|j| > \widetilde{M}}  \widetilde{a}^{(\eta\alpha/2^\ell)}_j \, e^{2 \pi i \ell \alpha j x} .
\end{equation}
Setting $V_\eta^{(\ell)} := [-\frac{1}{2} , \frac{1}{2}] \cap \supp \widetilde{p}_{\eta\alpha/2^\ell} (\ell \alpha x)$ for $\ell \in \N$, we obtain
\[
\begin{split}
&\int_{[-\frac{1}{2}, \frac{1}{2}] \backslash V_\eta^{(\ell)}}
\big| \widetilde{f}_{\eta\alpha/2^\ell,\Lambda_\ell} (x) \big|^2 \, dx \\
&\;\leq\;
\int_{-1/2}^{1/2}
\Big| \sum_{|j| > \widetilde{M}}  \widetilde{a}^{(\eta\alpha/2^\ell)}_j \, e^{2 \pi i \ell \alpha j x} \Big|^2 \, dx
=
\sum_{|j| > \widetilde{M}} \big| \widetilde{a}^{(\eta\alpha/2^\ell)}_j \big|^2 \\
&\;<\; \tfrac{\eta}{2^\ell}
\; \overset{\eqref{eqn:a-epsilonL-truncated-part}}{<} \;
R \, \tfrac{\eta}{2^\ell}
\sum_{j=-\widetilde{M}}^{\widetilde{M}} \big| \widetilde{a}^{(\eta\alpha/2^\ell)}_j \big|^2 .
\end{split}
\]
Note from \eqref{eqn:widetilde-p-epsilonL-time-limited} that
$\supp \widetilde{p}_{\eta\alpha/2^\ell} (\ell \alpha x)
= \frac{1}{\ell\alpha} \cup_{m \in \Z} \big(m + [-\frac{\eta\alpha}{4 \cdot 2^\ell}, \frac{\eta\alpha}{4 \cdot 2^\ell}] \big)
= \cup_{m \in \Z} \big(\frac{m}{\ell\alpha} + [-\frac{\eta}{4 \ell \cdot 2^\ell}, \frac{\eta}{4 \ell \cdot 2^\ell}] \big)$ which implies $| V_\eta^{(\ell)} | < \frac{\eta}{2^\ell}$.
Indeed, the set $[-\frac{1}{2} , \frac{1}{2}] \cap \cup_{m \in \Z} \big(m + [-\frac{\eta}{4 \cdot 2^\ell}, \frac{\eta}{4 \cdot 2^\ell}] \big) = [-\frac{\eta}{4 \cdot 2^\ell}, \frac{\eta}{4 \cdot 2^\ell}]$ is of length $\frac{\eta}{2^{\ell+1}}$, and the dilated set $\frac{1}{\ell} \cup_{m \in \Z} \big(m + [-\frac{\eta}{4 \cdot 2^\ell}, \frac{\eta}{4 \cdot 2^\ell}] \big) = \cup_{m \in \Z} \big(\frac{m}{\ell} + [-\frac{\eta}{4 \ell \cdot 2^\ell}, \frac{\eta}{4 \ell \cdot 2^\ell}] \big)$ restricted to $[-\frac{1}{2} , \frac{1}{2}]$ has Lebesgue measure $\frac{\eta}{2^{\ell+1}}$ as well, so the set $V_\eta^{(\ell)} = [-\frac{1}{2} , \frac{1}{2}] \cap \cup_{m \in \Z} \big(\frac{m}{\ell\alpha} + [-\frac{\eta}{4 \ell \cdot 2^\ell}, \frac{\eta}{4 \ell \cdot 2^\ell}] \big)$ with $0 < \alpha \leq 1$ has Lebesgue measure at most $\frac{\eta}{2^{\ell+1}}$ which is strictly less than $\frac{\eta}{2^\ell}$.
Finally, define $V_{\eta} := \cup_{\ell=1}^\infty V_\eta^{(\ell)}$ which clearly satisfies $|V_{\eta}| \leq \sum_{\ell=1}^\infty | V_\eta^{(\ell)} | < \sum_{\ell=1}^\infty \frac{\eta}{2^\ell} = \eta$.
Then for each $\ell \in \N$,
\[
\begin{split}
&\int_{[-\frac{1}{2}, \frac{1}{2}] \backslash V_{\eta}}
\Big| \sum_{j=-\widetilde{M}}^{\widetilde{M}} \widetilde{a}^{(\eta\alpha/2^\ell)}_j \, e^{2 \pi i s_{\ell} (j) x} \Big|^2 \, dx
=
\int_{[-\frac{1}{2}, \frac{1}{2}] \backslash V_{\eta}}
\big| \widetilde{f}_{\eta\alpha/2^\ell,\Lambda_\ell} (x) \big|^2 \, dx \\
&\leq
\int_{[-\frac{1}{2}, \frac{1}{2}] \backslash V_\eta^{(\ell)}}
\big| \widetilde{f}_{\eta\alpha/2^\ell,\Lambda_\ell} (x) \big|^2 \, dx
\;<\;
R \, \tfrac{\eta}{2^\ell}
\sum_{j=-\widetilde{M}}^{\widetilde{M}} \big| \widetilde{a}^{(\eta\alpha/2^\ell)}_j \big|^2
\end{split}
\]
which establishes the claim \eqref{eqn:our-claim-Z-2Z-3Z}.
This completes the proof.
\end{proof}

\begin{remark}[The construction of $S$ for $\Lambda_1 \,{=}\, \alpha \Z$, $\Lambda_2 \,{=}\, 2 \alpha \Z$, $\Lambda_3 \,{=}\, 3 \alpha \Z, \cdots$]
\label{rem:structure-of-S-Z-2Z-3Z}
\rm
In the proof above, the set $S$ is constructed as follows.
Given any $0 < \epsilon < 1$, choose an integer $R > \frac{1}{1-\epsilon}$ so that $0 < \epsilon < \frac{R-1}{R}$.
The set $S \subset [-\frac{1}{2}, \frac{1}{2}]$ is then given by $S := [-\frac{1}{2}, \frac{1}{2}] \backslash V$ with $V := \cup_{k=1}^\infty V_{\epsilon / 2^k}$, where
\begin{equation}\label{eqn:def-Veta-Vetaell-Z-2Z-3Z}
\begin{split}
& V_{\eta} \;:=\; \cup_{\ell=1}^\infty V_\eta^{(\ell)}
\qquad \text{and} \\
& V_\eta^{(\ell)} \;:=\; [-\tfrac{1}{2} , \tfrac{1}{2}] \cap \supp \widetilde{p}_{\eta\alpha/2^\ell} (\ell\alpha x) \\
& \qquad \,\, =\; [-\tfrac{1}{2} , \tfrac{1}{2}] \cap \tfrac{1}{\ell\alpha} \Big( \cup_{m \in \Z} \big(m + \big[ {-} \tfrac{\eta\alpha}{4 \cdot 2^\ell} , \tfrac{\eta\alpha}{4 \cdot 2^\ell} \big] \big) \Big) \\
& \qquad \,\, =\;  [-\tfrac{1}{2} , \tfrac{1}{2}] \cap \Big( \cup_{m \in \Z} \big(\tfrac{m}{\ell\alpha} + \big[ {-} \tfrac{\eta}{4 \ell \cdot 2^\ell} , \tfrac{\eta}{4 \ell \cdot 2^\ell} \big] \big) \Big) \\
& \,\, \text{for any} \;\; 0 < \eta < \tfrac{R-1}{R} \;\; \text{and} \;\; \ell \in \N .
\end{split}
\end{equation}
In short,
\[
\begin{split}
S :=\, & [-\tfrac{1}{2}, \tfrac{1}{2}] \backslash V
\quad \text{with}  \\
V :=\, & \cup_{k=1}^\infty V_{\epsilon / 2^k}
= \cup_{k=1}^\infty \cup_{\ell=1}^\infty V_{\epsilon/2^k}^{(\ell)} \\
=\, & [-\tfrac{1}{2} , \tfrac{1}{2}]
\cap \Big( \cup_{k=1}^\infty \cup_{\ell=1}^\infty \cup_{m \in \Z} \big(\tfrac{m}{\ell\alpha} + \big[ {-} \tfrac{\epsilon}{4 \ell \cdot 2^{k+\ell}} , \tfrac{\epsilon}{4 \ell \cdot 2^{k+\ell}} \big] \big) \Big) \\
=\, & [-\tfrac{1}{2} , \tfrac{1}{2}]
\cap \Big( \cup_{\ell=1}^\infty \cup_{m \in \Z} \big(\tfrac{m}{\ell\alpha} + \big[ {-} \tfrac{\epsilon}{4 \ell \cdot 2^{\ell+1}} , \tfrac{\epsilon}{4 \ell \cdot 2^{\ell+1}} \big] \big) \Big) ,
\end{split}
\]
where the set $V$ satisfies $|V| < \epsilon$ and thus $| S | > 1 - \epsilon$.
See Figure \ref{fig:SetV} for an illustration of the set $V$.
\end{remark}

We are now ready to prove Theorem \ref{thm:main-result-S-arbit-long-AP-fixed-common-diff}.
Note that the set $S \subset [-\frac{1}{2}, \frac{1}{2}]$ stated in Theorem \ref{thm:main-result-S-arbit-long-AP-fixed-common-diff} is exactly the resulting set of Proposition \ref{prop:seq-Z-2Z-3Z}, which is described in Remark \ref{rem:structure-of-S-Z-2Z-3Z}.

\begin{proof}[Proof of Theorem \ref{thm:main-result-S-arbit-long-AP-fixed-common-diff}]
Let $\Lambda \subset \R$ be a set containing
arbitrarily long arithmetic progressions with a fixed common difference $P \alpha$ for some $P \in \N$.
To prove that $E(\Lambda)$ is not a Riesz sequence in $L^2(S)$, it suffices to show that the set $V_{\eta} \subset [-\frac{1}{2}, \frac{1}{2}]$ given by \eqref{eqn:def-Veta-Vetaell-Z-2Z-3Z} for $0 < \eta < \tfrac{R-1}{R}$ and $\ell \in \N$ (with a fixed integer $R > \frac{1}{1-\epsilon}$), satisfies the following property:
there is a finitely supported sequence $b^{(\eta,\Lambda)} = \{ b^{(\eta,\Lambda)}_\lambda \}_{\lambda \in \Lambda}$ with
\begin{equation}\label{eqn:our-claim-rationally-periodic-Riesz-spectrum}
\int_{[-\frac{1}{2}, \frac{1}{2}] \backslash V_{\eta}}
\Big| \sum_{\lambda \in \Lambda} b^{(\eta,\Lambda)}_\lambda \, e^{2 \pi i \lambda x} \Big|^2 \, dx
\;\leq\;
R \, \tfrac{\eta}{2^P}
\sum_{\lambda \in \Lambda} \big| b^{(\eta,\Lambda)}_\lambda \big|^2  .
\end{equation}
Indeed, since
$S := [-\frac{1}{2}, \frac{1}{2}] \backslash \cup_{k=1}^\infty V_{\epsilon / 2^k}$ (see Remark \ref{rem:structure-of-S-Z-2Z-3Z}), it then holds for any $k \in \N$,
\[
\begin{split}
\int_S
\Big| \sum_{\lambda \in \Lambda} b^{(\epsilon / 2^k,\Lambda)}_\lambda \, e^{2 \pi i \lambda x} \Big|^2 \, dx
&\;\leq\;
\int_{[-\frac{1}{2}, \frac{1}{2}] \backslash V_{\epsilon / 2^k}}
\Big| \sum_{\lambda \in \Lambda} b^{(\epsilon / 2^k,\Lambda)}_\lambda \, e^{2 \pi i \lambda x} \Big|^2 \, dx \\
&\overset{\eqref{eqn:our-claim-rationally-periodic-Riesz-spectrum}}{\leq} \;
R \, \tfrac{\epsilon}{2^{k+P}}
\sum_{\lambda \in \Lambda} \big| b^{(\epsilon / 2^k,\Lambda)}_\lambda \big|^2
\end{split}
\]
which implies that $E(\Lambda)$ is not a Riesz sequence in $L^2(S)$.

To prove the claim \eqref{eqn:our-claim-rationally-periodic-Riesz-spectrum}, consider the sequence $\widetilde{a}^{(\eta\alpha/2^P)} = \{ \widetilde{a}^{(\eta\alpha/2^P)}_j \}_{j \in \Z} \in \ell_2(\Z)$, the function $\widetilde{p}_{\eta\alpha/2^P}$, and the number $\widetilde{M} = \widetilde{M}(\eta\alpha/2^P) \in \N$ taken respectively from \eqref{eqn:def-widetilde-a}-\eqref{eqn:a-epsilonL-truncated-part} with $\ell = P$.
By the assumption, the set $\Lambda \subset \R$ contains an arithmetic progression of length $2 \widetilde{M} {+} 1$ with common difference $P\alpha$, which can be expressed as
\[
s_{\Lambda} (j) := P \alpha j + d,
\qquad j = -\widetilde{M}, \ldots, \widetilde{M} ,
\]
for some $d \in \Z$.
Similarly as in \eqref{eqn:def-f-epsilon-ell-Z-2Z-3Z} and \eqref{eqn:f-eps-ell-minus-p-eps-ellx-UpperBound-by-eps-Z-2Z-3Z},
we define
\[
\widetilde{f}_{\Lambda} (x)
:= \sum_{j=-\widetilde{M}}^{\widetilde{M}} \widetilde{a}^{(\eta\alpha/2^P)}_j \, e^{2 \pi i s_{\Lambda} (j) x}
\quad
\text{for} \;\; x \in \R
\]
and observe that
\[
\widetilde{f}_{\Lambda} (x) - \widetilde{p}_{\eta\alpha/2^P} (P\alpha x) \, e^{2 \pi i d x}
=
- \sum_{|j| > \widetilde{M}}  \widetilde{a}^{(\eta\alpha/2^P)}_j \, e^{2 \pi i (P\alpha j + d) x}
\quad \text{for all} \;\; x \in \R .
\]
Recalling that $V_\eta^{(P)} := [-\frac{1}{2} , \frac{1}{2}] \cap \supp \widetilde{p}_{\eta\alpha/2^P} (P\alpha x)$ (see Remark \ref{rem:structure-of-S-Z-2Z-3Z}), we have
\[
\begin{split}
&\int_{[-\frac{1}{2}, \frac{1}{2}] \backslash V_\eta^{(P)}}
\big| \widetilde{f}_{\Lambda} (x) \big|^2 \, dx \\
&\leq
\int_{-1/2}^{1/2}
\Big| \sum_{|j| > \widetilde{M}}  \widetilde{a}^{(\eta\alpha/2^P)}_j \, e^{2 \pi i (P\alpha j + d) x} \Big|^2 \, dx
=
\sum_{|j| > \widetilde{M}} \big| \widetilde{a}^{(\eta\alpha/2^P)}_j \big|^2 \\
&< \tfrac{\eta}{2^P}
\;<\;
R \, \tfrac{\eta}{2^P}
\end{split}
\]
where the inequality \eqref{eqn:a-epsilonL-truncated-part} for $\ell = P$ is used in the last step.
Since $V_{\eta} := \cup_{\ell=1}^\infty V_\eta^{(\ell)}$, we have
\[
\begin{split}
\int_{[-\frac{1}{2}, \frac{1}{2}] \backslash V_{\eta}}
\Big| \sum_{j=-\widetilde{M}}^{\widetilde{M}} \widetilde{a}^{(\eta\alpha/2^P)}_j \, e^{2 \pi i s_{\Lambda} (j) x} \Big|^2 \, dx
&=
\int_{[-\frac{1}{2}, \frac{1}{2}] \backslash V_{\eta}}
\big| \widetilde{f}_{\Lambda} (x) \big|^2 \, dx \\
&\leq
\int_{[-\frac{1}{2}, \frac{1}{2}] \backslash V_\eta^{(P)}}
\big| \widetilde{f}_{\Lambda} (x) \big|^2 \, dx
\;<\;
R \, \tfrac{\eta}{2^P}
\end{split}
\]
which establishes the claim \eqref{eqn:our-claim-rationally-periodic-Riesz-spectrum}.
\end{proof}

\section{Proof of Theorem \ref{thm:Lambda1-Lambda2-Lambda3-introduction}}
\label{sec:proof-second-result}

We will now prove Theorem \ref{thm:Lambda1-Lambda2-Lambda3-introduction} which generalizes Proposition \ref{prop:seq-Z-2Z-3Z} from
$\Lambda_1 \,{=}\, \alpha \Z$, $\Lambda_2 \,{=}\, 2 \alpha \Z$, $\Lambda_3 \,{=}\, 3 \alpha \Z, \cdots$ to arbitrary separated sets $\Lambda_1, \Lambda_2, \ldots \subset \R$ with positive upper Beurling density.
The proof is similar to the proof of Proposition \ref{prop:seq-Z-2Z-3Z}, but
since an arbitrary separated set is in general non-periodic we need the additional step of extracting an \emph{approximate} arithmetic progression from each set $\Lambda_\ell$ with the help of Lemma \ref{lem:Lemma5-1-in-OU08}.

\begin{proof}[Proof of Theorem \ref{thm:Lambda1-Lambda2-Lambda3-introduction}]
Fix any $0 < \epsilon < 1$ and choose an integer $R > \frac{1}{1-\epsilon}$ so that $0 < \epsilon < \frac{R-1}{R}$.
We claim that for each $0 < \eta < \frac{R-1}{R}$ there exists a set $V_{\eta} \subset [-\frac{1}{2}, \frac{1}{2}]$ with $| V_{\eta} | < \eta$ satisfying the following property: for each $\ell \in \N$ there is a finitely supported sequence $b^{(\eta,\Lambda_\ell)} = \{ b^{(\eta,\Lambda_\ell)}_\lambda \}_{\lambda \in \Lambda_\ell}$ with
\begin{equation}\label{eqn:our-claim-Lambda1-Lambda2-Lambda3}
\int_{[-\frac{1}{2}, \frac{1}{2}] \backslash V_{\eta}}
\Big| \sum_{\lambda \in \Lambda_\ell} b^{(\eta,\Lambda_\ell)}_\lambda \, e^{2 \pi i \lambda x} \Big|^2 \, dx
\;\leq\; R \, \eta^2  \sum_{\lambda \in \Lambda_\ell} \big| b^{(\eta,\Lambda_\ell)}_\lambda \big|^2  .
\end{equation}

To prove the claim \eqref{eqn:our-claim-Lambda1-Lambda2-Lambda3}, fix any $0 < \eta < \frac{R-1}{R}$.
For each $\ell \in \N$,
let $a^{(\eta/2^\ell)} = \{ a^{(\eta/2^\ell)}_j \}_{j \in \Z}$ be an $\ell_1$-sequence with unit $\ell_2$-norm $\| a^{(\eta/2^\ell)} \|_{\ell_2} = 1$ such that
\begin{equation}\label{eqn:p-epsilonL-time-limited-OriginalEll1seqMethod}
p_{\eta/2^\ell}(x) := \sum_{j \in \Z} a^{(\eta/2^\ell)}_j \, e^{2 \pi i j x}
\quad \text{satisfies} \quad
p_{\eta/2^\ell}(x) = 0
\quad \text{for} \;\; \tfrac{\eta}{4 \cdot 2^\ell} \leq |x| \leq \tfrac{1}{2} .
\end{equation}
Since the sequence $a^{(\eta/2^\ell)} \in \ell_1(\Z)$ is not finitely supported, there is a number $M = M(\eta/2^\ell) \in \N$ with $0 < \sum_{|j| > M} |a^{(\eta/2^\ell)}_j| < \frac{\eta}{2^\ell}$. Note that since $|a^{(\eta/2^\ell)}_j| \leq \| a^{(\eta/2^\ell)} \|_{\ell_2} =1$ for all $j \in \Z$, we have
\[
\sum_{|j| > M} \big| a^{(\eta/2^\ell)}_j \big|^2
= \sum_{|j| > M} \big| a^{(\eta/2^\ell)}_j \big| < \tfrac{\eta}{2^\ell} ,
\]
so that
\begin{equation}\label{eqn:a-epsilonL-truncated-part-OriginalEll1seqMethod}
\sum_{j=-M}^M \big| a^{(\eta/2^\ell)}_j \big|^2
\;>\; 1 - \tfrac{\eta}{2^\ell}
\;\geq\; 1 - \eta
\;>\; \tfrac{1}{R} .
\end{equation}
We then choose a small parameter $0 < \delta = \delta (\eta/2^\ell) < 1$ satisfying
\[
\sin ( \pi \delta / 2)  <  \frac{\eta/2^\ell}{2 \sum_{j=-M}^M | a^{(\eta/2^\ell)}_j | } \; ,
\]
so that $\sum_{j=-M}^M | a^{(\eta/2^\ell)}_j | \cdot | e^{i \pi \delta} -  1 | = \sum_{j=-M}^M | a^{(\eta/2^\ell)}_j | \cdot 2 \sin ( \pi \delta / 2)  < \frac{\eta}{2^\ell}$.
Note that all the terms up to this point depend only on the parameters $\eta$ and $\ell$, in fact, only on the value $\eta/2^\ell$.

Now, the set $\Lambda_\ell$ comes into play.
Applying Lemma \ref{lem:Lemma5-1-in-OU08} to the set $\Lambda_\ell$ with the parameters $M$ and $\delta$ chosen above, we deduce that there exist constants $c = c(\eta/2^\ell,\Lambda_\ell) \in \N$ and $d = d(\eta/2^\ell,\Lambda_\ell) \in \R$, and an increasing sequence $s_{\eta/2^\ell,\Lambda_\ell} (-M) < s_{\eta/2^\ell,\Lambda_\ell} (-M+1) < \ldots < s_{\eta/2^\ell,\Lambda_\ell} (M)$ in $\Lambda_\ell$ satisfying
\[
\big| s_{\eta/2^\ell,\Lambda_\ell} (j) - c j - d \big| \leq \delta
\quad \text{for} \;\; j = -M, \ldots, M .
\]
For $x \in [-\frac{1}{2}, \frac{1}{2}]$, we define
\[
f_{\eta/2^\ell,\Lambda_\ell} (x)
:= \sum_{j=-M}^M a^{(\eta/2^\ell)}_j \, \exp \big( 2 \pi i s_{\eta/2^\ell,\Lambda_\ell} (j) x \big)
\]
and observe that
\[
\begin{split}
&\big| f_{\eta/2^\ell,\Lambda_\ell} (x) - p_{\eta/2^\ell} ( c x) \, e^{2 \pi i d x}  \big| \\
&\leq \Big| \sum_{j=-M}^M   a^{(\eta/2^\ell)}_j \Big( \exp \big( 2 \pi i s_{\eta/2^\ell,\Lambda_\ell} (j) x \big) -  e^{2 \pi i (cj+d) x} \Big) \Big|  + \Big| \sum_{|j| > M}  a^{(\eta/2^\ell)}_j \, e^{2 \pi i (cj+d) x} \Big|  \\
&\leq \sum_{j=-M}^M \big| a^{(\eta/2^\ell)}_j \big| \cdot \Big| \exp \Big( 2 \pi i \big( s_{\eta/2^\ell,\Lambda_\ell} (j) - cj-d \big) x \Big) -  1 \Big| + \sum_{|j| > M} \big| a^{(\eta/2^\ell)}_j \big|  \\
&< \tfrac{\eta}{2^\ell} + \tfrac{\eta}{2^\ell} = \tfrac{\eta}{2^{\ell-1}}
\;\leq\; \eta .
\end{split}
\]
Setting $V_\eta^{(\ell)} := [-\frac{1}{2} , \frac{1}{2}] \cap \supp p_{\eta/2^\ell} \big( c(\eta/2^\ell,\Lambda_\ell) x \big)$,
we have
\[
\int_{[-\frac{1}{2}, \frac{1}{2}] \backslash V_\eta^{(\ell)}}
\big| f_{\eta/2^\ell,\Lambda_\ell} (x) \big|^2 \, dx
\;\leq\; \eta^2
\;\overset{\eqref{eqn:a-epsilonL-truncated-part-OriginalEll1seqMethod}}{<}\;
R \, \eta^2
\sum_{j=-M}^M \big| a^{(\eta/2^\ell)}_j \big|^2 .
\]
Similarly as in the proof of Proposition \ref{prop:seq-Z-2Z-3Z},
we have $| V_\eta^{(\ell)} | < \frac{\eta}{2^\ell}$ and therefore
the set $V_{\eta} := \cup_{\ell=1}^\infty V_\eta^{(\ell)}$ satisfies $|V_{\eta}| < \eta$.
It then holds for each $\ell \in \N$,
\[
\begin{split}
& \int_{[-\frac{1}{2}, \frac{1}{2}] \backslash V_{\eta}}
\Big| \sum_{j=-M}^M a^{(\eta/2^\ell)}_j \, \exp \big( 2 \pi i s_{\eta/2^\ell,\Lambda_\ell} (j) x \big) \Big|^2 \, dx
\;=\;
\int_{[-\frac{1}{2}, \frac{1}{2}] \backslash V_{\eta}}
\big| f_{\eta/2^\ell,\Lambda_\ell} (x) \big|^2 \, dx \\
&\leq
\int_{[-\frac{1}{2}, \frac{1}{2}] \backslash V_\eta^{(\ell)}}
\big| f_{\eta/2^\ell,\Lambda_\ell} (x) \big|^2 \, dx
\;<\;
R \, \eta^2
\sum_{j=-M}^{M} \big| a^{(\eta/2^\ell)}_j \big|^2
\end{split}
\]
which proves the claim \eqref{eqn:our-claim-Lambda1-Lambda2-Lambda3}.

Finally, based on the established claim \eqref{eqn:our-claim-Lambda1-Lambda2-Lambda3} we define
$V := \cup_{k=1}^\infty V_{\epsilon / 2^k}$ and $S := [-\frac{1}{2}, \frac{1}{2}] \backslash V$.
Clearly, we have $|V| \leq \sum_{k=1}^\infty | V_{\epsilon / 2^k} | < \sum_{k=1}^\infty \frac{\epsilon}{2^k} = \epsilon$, so that $| S | > 1 - \epsilon$.
Also, it holds for any $k, \ell \in \N$,
\[
\begin{split}
\int_S  \Big| \sum_{\lambda \in \Lambda_\ell} b^{(\epsilon/2^k,\Lambda_\ell)}_\lambda \, e^{2 \pi i \lambda x}
\Big|^2 \, dx
&\;\leq\;
\int_{[-\frac{1}{2}, \frac{1}{2}] \backslash V_{\epsilon / 2^k}}
\Big| \sum_{\lambda \in \Lambda_\ell} b^{(\epsilon/2^k,\Lambda_\ell)}_\lambda \, e^{2 \pi i \lambda x} \Big|^2 \, dx \\
&\overset{\eqref{eqn:our-claim-Lambda1-Lambda2-Lambda3}}{\leq} \; R \, \big( \tfrac{\epsilon}{2^k} \big)^2
\sum_{\lambda \in \Lambda_\ell} \big| b^{(\epsilon/2^k,\Lambda_\ell)}_\lambda \big|^2 .
\end{split}
\]
By fixing any $\ell \in \N$ and letting $k \rightarrow \infty$, we conclude that $E(\Lambda_{\ell})$ is not a Riesz sequence in $L^2(S)$.
\end{proof}

\begin{remark}[The construction of $S$ for arbitrary separated sets $\Lambda_1, \Lambda_2, \ldots \subset \R$]
\label{rem:structure-of-S-Lambda1-Lambda2-Lambda3}
\rm
In the proof above, the set $S$ for arbitrary separated sets $\Lambda_1, \Lambda_2, \ldots \subset \R$ is constructed as follows.
Given any $0 < \epsilon < 1$, choose an integer $R > \frac{1}{1-\epsilon}$ so that $0 < \epsilon < \frac{R-1}{R}$.
The set $S \subset [-\frac{1}{2}, \frac{1}{2}]$ is then given by $S := [-\frac{1}{2}, \frac{1}{2}] \backslash V$ with $V := \cup_{k=1}^\infty V_{\epsilon / 2^k}$, where
\begin{equation*}
\begin{split}
& V_{\eta} \;:=\; \cup_{\ell=1}^\infty V_\eta^{(\ell)}
\qquad \text{and} \\
& V_\eta^{(\ell)} \;:=\; [-\tfrac{1}{2} , \tfrac{1}{2}] \cap \supp p_{\eta/2^\ell} \big( c(\eta/2^\ell,\Lambda_\ell) x \big)
\;=\; [-\tfrac{1}{2} , \tfrac{1}{2}] \cap \Big( \tfrac{1}{c(\eta/2^\ell,\Lambda_\ell)} \, \supp p_{\eta/2^\ell} \Big) \\
& \qquad \,\, \subset\;  [-\tfrac{1}{2} , \tfrac{1}{2}] \cap \Big( \cup_{m \in \Z} \big(\tfrac{m}{c(\eta/2^\ell,\Lambda_\ell)} + \big[ {-} \tfrac{\eta}{4 \cdot c(\eta/2^\ell,\Lambda_\ell) \cdot 2^\ell} , \tfrac{\eta}{4 \cdot c(\eta/2^\ell,\Lambda_\ell) \cdot 2^\ell} \big] \big) \Big) \\
& \,\, \text{for any} \;\; 0 < \eta < \tfrac{R-1}{R} \;\; \text{and} \;\; \ell \in \N .
\end{split}
\end{equation*}
Here, $c(\eta/2^\ell,\Lambda_\ell)$ is a positive integer depending on the value $\eta/2^\ell$ and the set $\Lambda_\ell$. In short,
\[
\begin{split}
S :=\, & [-\tfrac{1}{2}, \tfrac{1}{2}] \backslash V
\quad \text{with}  \\
V :=\, & \cup_{k=1}^\infty V_{\epsilon / 2^k}
\;=\; \cup_{k=1}^\infty \cup_{\ell=1}^\infty V_{\epsilon/2^k}^{(\ell)} \\
=\, & [-\tfrac{1}{2} , \tfrac{1}{2}] \cap \Big( \cup_{k=1}^\infty \cup_{\ell=1}^\infty \, \tfrac{1}{c(\epsilon/2^{k+\ell},\Lambda_\ell)} \, \supp p_{\epsilon/2^{k+\ell}} \Big) \\
\subset\, & [-\tfrac{1}{2} , \tfrac{1}{2}] \cap \\
& \Big( \cup_{k=1}^\infty \cup_{\ell=1}^\infty \cup_{m \in \Z}
\big(\tfrac{m}{c(\epsilon/2^{k+\ell},\Lambda_\ell)} + \big[ {-} \tfrac{\epsilon}{4 \cdot c(\epsilon/2^{k+\ell},\Lambda_\ell) \cdot 2^{k+\ell}} , \tfrac{\epsilon}{4 \cdot c(\epsilon/2^{k+\ell},\Lambda_\ell) \cdot 2^{k+\ell}} \big] \big) \Big) .
\end{split}
\]
Recall that we were able to eliminate the union $\cup_{k=1}^\infty$ in the expression of $V$ in Remark \ref{rem:structure-of-S-Z-2Z-3Z},
because for any fixed $\ell \in \N$ the sets $\big[ {-} \tfrac{\epsilon}{4 \ell \cdot 2^{k+\ell}} , \tfrac{\epsilon}{4 \ell \cdot 2^{k+\ell}} \big]$, $k=1,2, \ldots$ are decreasingly nested.
Unfortunately, the trick cannot be applied here even if $\supp p_{\epsilon/2^{k+\ell}} \cap [-\frac{1}{2} , \frac{1}{2}] = \big[ {-} \tfrac{\epsilon}{4 \cdot 2^{k+\ell}} , \tfrac{\epsilon}{4 \cdot 2^{k+\ell}} \big]$ for all $k, \ell \in \N$ and the numbers $c(\epsilon/2^{k+\ell},\Lambda_\ell)$, $k =1,2,\ldots$ increase by factors of positive integers (exploiting the `moreover' part of  Lemma \ref{lem:Lemma5-1-in-OU08}) for $\ell \in \N$ fixed, in which case the sets $\big[ {-} \tfrac{\epsilon}{4 \cdot c(\epsilon/2^{k+\ell},\Lambda_\ell) \cdot 2^{k+\ell}} , \tfrac{\epsilon}{4 \cdot c(\epsilon/2^{k+\ell},\Lambda_\ell) \cdot 2^{k+\ell}} \big]$, $k =1,2,\ldots$  are decreasingly nested for $\ell \in \N$ fixed.
This is because the period $\frac{1}{c(\epsilon/2^{k+\ell},\Lambda_\ell)}$ of the periodization involved with the union $\cup_{m \in \Z}$, depends also on $k$.
\end{remark}

\section{Remarks}
\label{sec:Conclusion}

Let us discuss some obstacles in extending our results (Theorems \ref{thm:main-result-S-arbit-long-AP-fixed-common-diff} and \ref{thm:Lambda1-Lambda2-Lambda3-introduction}) to the class of \emph{arbitrary} separated sets $\Lambda \subset \R$ with positive upper Beurling density.
Our result relies on the proof technique of Olevskii and Ulanovskii \cite{OU08} which is based on the celebrated Szemer\'{e}di's theorem \cite{Sz75} stating that
\[
  \parbox{\dimexpr\linewidth-4em}{%
    \strut
    \it
    any integer set $\Omega \subset \Z$ with positive upper Beurling density $D^+ (\Omega) > 0$ contains at least one arithmetic progression of length $M$ for all $M \in \N$.
    \strut
  }
\]
If it were even true that for any integer set $\Omega \subset \Z$ with $D^+ (\Omega) > 0$,
\begin{equation}\tag{$\ast$}\label{eq:ast-wrong-number-theory-statement}
  \parbox{\dimexpr\linewidth-5em}{%
    \strut
    \it
    there exists a number $P \in \N$ such that $\Omega$ contains arbitrarily long arithmetic progressions with common difference $P$,
    \strut
  }
\end{equation}
then Theorem \ref{thm:main-result-S-arbit-long-AP-fixed-common-diff} would imply a stronger result:
\begin{equation}\label{eq:wrong-strengthen}
  \parbox{\dimexpr\linewidth-5em}{%
    \strut
    \it
    Let $S \subset [-\frac{1}{2}, \frac{1}{2}]$ be the set given by \eqref{eqn:first-thm-def-set-S} with $\alpha = 1$ and any $0 < \epsilon < 1$.
    If $\Lambda \subset \R$ is a separated set with $D^+ (\Lambda) > 0$, then $E(\Lambda)$ is not a Riesz sequence in $L^2(S)$.
    \strut
  }
\end{equation}
To see this, suppose to the contrary that $\Lambda = \{ \lambda_n \}_{n \in \Z} \subset \R$ is a separated set with $D^+ (\Lambda) > 0$ such that $E(\Lambda)$ is a Riesz sequence in $L^2(S)$.
Then according to Theorem \ref{thm:PaleyWienerStability} there is a constant $\theta = \theta(\Lambda,S) > 0$ such that
$E(\Lambda')$ is a Riesz sequence in $L^2(S)$ whenever $\Lambda' = \{ \lambda_n' \}_{n \in \Z} \subset \R$ satisfies
$| \lambda_n' - \lambda_n | \leq \theta$ for all $n \in \Z$.
This allows for a replacement of the set $\Lambda \subset \R$ with its perturbation $\Lambda' \subset \frac{1}{N} \Z$ for some large $N \in \N$.
Certainly, the set $N \Lambda' \subset \Z$ satisfies $D^+ (N \Lambda') > 0$, and thus \eqref{eq:ast-wrong-number-theory-statement} would imply that there is a number $P \in \N$ such that $N \Lambda'$ contains arbitrarily long arithmetic progressions with common difference $P$.
In turn, the set $\Lambda'$ would also contain arbitrarily long arithmetic progressions with common difference $P$,
and therefore $E(\Lambda')$ would not be a Riesz sequence in $L^2(S)$ by Theorem \ref{thm:main-result-S-arbit-long-AP-fixed-common-diff},
yielding a contradiction.

Unfortunately, as shown in the following example, there exist some sets $\Omega \subset \Z$ with $D^+ (\Omega) > 0$ which do not satisfy \eqref{eq:ast-wrong-number-theory-statement}.
Hence, the improvement of Theorem \ref{thm:main-result-S-arbit-long-AP-fixed-common-diff} to \eqref{eq:wrong-strengthen} does not work as we wished.

\begin{example}
\rm
As mentioned in \cite[Section 1]{CM20}, very little is known about the integer sets that satisfy \eqref{eq:ast-wrong-number-theory-statement}, i.e., the integer sets containing arbitrarily long arithmetic progressions with fixed common difference.
Motivated by a discussion in \cite{CM20},
we will now provide a set $\Omega \subset \Z$ with positive uniform Beurling density which does not satisfy \eqref{eq:ast-wrong-number-theory-statement}.

Consider the set $\Omega := \pm \Omega_0 = \Omega_0 \cup (-\Omega_0)$ where
$\Omega_0 := \{ 1, 2, 3, 5, 6, 7, 10, 11, 13$, $14, 15, 17, \ldots \}$
is the set of square-free integers, i.e., the integers that are not divisible by $n^2$ for $n \in \N$ prime.
It is well-known (see e.g., \cite{Wa63}) that
$\lim_{r \rightarrow \infty} | \Omega_0 \cap \{ 1, 2, \ldots, r \} | / r
=
\frac{6}{\pi^2}
\approx 0.6079$, which implies $D (\Omega) = \frac{6}{\pi^2}$.
Note that if $\Omega$ contains a long arithmetic progression, then either $-\Omega_0$ or $\Omega_0$ contains at least half portion of that arithmetic progression.
By symmetry, this implies that if \eqref{eq:ast-wrong-number-theory-statement} holds for $\Omega$, then it holds also for $\Omega_0$.
Thus, to prove that \eqref{eq:ast-wrong-number-theory-statement} does not hold for $\Omega$, it will be enough to show that \eqref{eq:ast-wrong-number-theory-statement} does not hold for $\Omega_0$.

Suppose to the contrary that \eqref{eq:ast-wrong-number-theory-statement} holds for $\Omega_0$. Then
there is a number $P \in \N$ such that $\Omega_0$ contains an arithmetic progression of length $Q^2$ with common difference $P$, where $Q \in \N$ is any prime number greater than $P$,
that is,
\[
d, \; d{+}P , \; d{+}2P, \; \ldots, \; d{+}(Q^2 {-} 1)P
\;\; \in \; \Omega_0
\quad \text{for some} \;\; d \in \N .
\]
Since $Q$ is prime and $P < Q$, we have $\gcd (P, Q^2) = 1$ which implies that all the numbers $d + jP$ for $j=0,\ldots,Q^2-1$ have distinct residues modulo $Q^2$.
In particular, there is a number $d + j P \in \Omega_0$ which is divisible by $Q^2$, contradicting with the choice of $\Omega_0$.
Hence, the property \eqref{eq:ast-wrong-number-theory-statement} does not hold for $\Omega_0$ and thus neither for $\Omega$.
\end{example}

\begin{remark}
\label{rem:PowerOf4}
\rm
It is possible to slightly improve the set $S \subset [-\frac{1}{2}, \frac{1}{2}]$ appearing in the statement of Theorem \ref{thm:main-result-S-arbit-long-AP-fixed-common-diff} (see Remark \ref{rem:structure-of-S-Z-2Z-3Z} for the construction of $S$).
Instead of the set $S$ in Theorem \ref{thm:main-result-S-arbit-long-AP-fixed-common-diff},
consider the set
$S := [-\frac{1}{2}, \frac{1}{2}] \backslash V$ with $V := \cup_{k=1}^\infty V_{\epsilon / 2^k}$, where
\begin{equation}\label{eqn:def-Veta-Vetaell-Z-2Z-3Z-improvement}
\begin{split}
& V_{\eta} \;:=\; \cup_{\ell=1}^\infty V_\eta^{(\ell)}
\qquad \text{and} \\
& V_\eta^{(\ell)} \;:=\;
[-\tfrac{1}{2} , \tfrac{1}{2}] \cap \Big( \textstyle \bigcup_{r=1}^{2^\ell} \supp \widetilde{p}_{\eta\alpha/4^\ell} \big( c^{(\ell)}_r \alpha x \big) \Big) \\
& \qquad \,\, =\;  [-\tfrac{1}{2} , \tfrac{1}{2}] \cap \textstyle \bigcup_{r=1}^{2^\ell} \Big( \tfrac{1}{c^{(\ell)}_r \alpha} \textstyle \bigcup_{m \in \Z} \big(m + \big[ {-} \tfrac{\eta\alpha}{4^{\ell+1}} , \tfrac{\eta\alpha}{4^{\ell+1}} \big] \big) \Big) \\
& \qquad \,\, =\;  [-\tfrac{1}{2} , \tfrac{1}{2}] \cap \textstyle \bigcup_{r=1}^{2^\ell} \Big( \textstyle \bigcup_{m \in \Z} \big(\tfrac{m}{c^{(\ell)}_r \alpha} + \big[ {-} \tfrac{\eta}{c^{(\ell)}_r \cdot 4^{\ell+1}} , \tfrac{\eta}{ c^{(\ell)}_r \cdot 4^{\ell+1}} \big] \big) \Big) \\
& \,\, \text{for any} \;\; 0 < \eta < \tfrac{R-1}{R} \;\; \text{and} \;\; \ell \in \N .
\end{split}
\end{equation}
Here, $c^{(\ell)} = \{ c^{(\ell)}_r \}_{r=1}^{2^\ell} \in \N^{2^\ell}$ is any $\N$-valued vector of size $2^\ell$ (for instance, in the light of the set $\mathcal{C}$ discussed before Lemma \ref{lem:Lemma5-1-in-OU08}, one may choose $c^{(\ell)}_r = r$ or $c^{(\ell)}_r = r^{1000}$ for $r = 1, 2, 3, \ldots, 2^\ell$), and
$\widetilde{p}_{\eta\alpha/4^\ell}$ is the $1$-periodic function given by
\[
\widetilde{p}_{\eta\alpha/4^\ell}(x)
=
\begin{cases}
\sqrt{\tfrac{2 \cdot 4^{\ell}}{\eta\alpha}} & \text{for} \;\; x \in \big[ {-} \tfrac{\eta\alpha}{4^{\ell+1}} , \tfrac{\eta\alpha}{4^{\ell+1}} \big] , \\[.5em]
0 & \text{for} \;\; x \in \big[ {-} \tfrac{1}{2} , \tfrac{1}{2} \big) \big\backslash \big[ {-} \tfrac{\eta\alpha}{4^{\ell+1}} , \tfrac{\eta\alpha}{4^{\ell+1}} \big] ,
\end{cases}
\]
which is consistent with the notation of $\widetilde{p}_{\eta\alpha/2^\ell}(x)$ in \eqref{eqn:widetilde-p-epsilonL-time-limited}.
In short,
\[
\begin{split}
& S := [-\tfrac{1}{2}, \tfrac{1}{2}] \backslash V
\quad \text{with} \\
& V :=\, \cup_{k=1}^\infty V_{\epsilon / 2^k}
= \cup_{k=1}^\infty \cup_{\ell=1}^\infty V_{\epsilon/2^k}^{(\ell)} \\
& \;\;\;\,=\, [-\tfrac{1}{2} , \tfrac{1}{2}] \cap \Big( \textstyle \bigcup_{k=1}^\infty \textstyle \bigcup_{\ell=1}^\infty \textstyle \bigcup_{r=1}^{2^\ell}
\textstyle \bigcup_{m \in \Z} \big(\tfrac{m}{c^{(\ell)}_r \alpha} + \big[ {-} \tfrac{\epsilon}{c^{(\ell)}_r \cdot 2^k \cdot 4^{\ell+1}} , \tfrac{\epsilon}{ c^{(\ell)}_r \cdot 2^k \cdot 4^{\ell+1}} \big] \big) \Big)  \\
& \;\;\;\,=\, [-\tfrac{1}{2} , \tfrac{1}{2}] \cap \Big( \textstyle \bigcup_{\ell=1}^\infty \textstyle \bigcup_{r=1}^{2^\ell}
\textstyle \bigcup_{m \in \Z} \big(\tfrac{m}{c^{(\ell)}_r \alpha} + \big[ {-} \tfrac{\epsilon}{c^{(\ell)}_r \cdot 2 \cdot 4^{\ell+1}} , \tfrac{\epsilon}{ c^{(\ell)}_r \cdot 2 \cdot 4^{\ell+1}} \big] \big) \Big)  .
\end{split}
\]
Note that for each $\ell \in \N$, we have $\big| [-\frac{1}{2} , \frac{1}{2}] \cap \supp \widetilde{p}_{\eta\alpha/4^\ell} \big( c^{(\ell)}_r \alpha x \big) \big| < \frac{\eta}{4^\ell}$ for all $r = 1, 2, 3, \ldots, 2^\ell$, regardless of the choice of $c^{(\ell)} = \{ c^{(\ell)}_r \}_{r=1}^{2^\ell} \in \N^{2^\ell}$,
which then implies
$| V_\eta^{(\ell)} | < 2^\ell \,{\cdot}\, \tfrac{\eta}{4^\ell} = \tfrac{\eta}{2^\ell}$.
In turn, we have $|V_{\eta}| \leq \sum_{\ell=1}^\infty | V_\eta^{(\ell)} | < \sum_{\ell=1}^\infty \frac{\eta}{2^\ell} = \eta$ and consequently,
$|V| \leq \sum_{k=1}^\infty | V_{\epsilon / 2^k} | < \sum_{k=1}^\infty \frac{\epsilon}{2^k} = \epsilon$ and $| S | > 1 - \epsilon$.

An inspection of the proof of Proposition \ref{prop:seq-Z-2Z-3Z} and Theorem \ref{thm:main-result-S-arbit-long-AP-fixed-common-diff} shows that
the original set $V_\eta^{(\ell)} := [-\tfrac{1}{2} , \tfrac{1}{2}] \cap \supp \widetilde{p}_{\eta\alpha/2^\ell} (\ell\alpha x)$ defined in \eqref{eqn:def-Veta-Vetaell-Z-2Z-3Z}
can accommodate all the arithmetic progressions with common difference `$\ell\alpha$' through the function $\widetilde{p}_{\eta\alpha/2^\ell} (\ell\alpha x)$ whose dilation factor is `$\ell\alpha$'.
Defining $V_\eta^{(\ell)}$ as in \eqref{eqn:def-Veta-Vetaell-Z-2Z-3Z-improvement}, on the other hand,
allows for a multiple choice of common difference parameter `$P \alpha$' with $P \in \{ c^{(\ell)}_r : r = 1, 2, 3, \ldots, 2^\ell \}$,
which is to be used as the dilation factor associated with the function $\widetilde{p}_{\eta\alpha/4^\ell}$.
Accordingly, the new set $V_\eta^{(\ell)}$ can accommodate all the arithmetic progressions with common difference `$c^{(\ell)}_r \alpha$' for $r = 1, 2, 3, \ldots, 2^\ell$, through the function $\widetilde{p}_{\eta\alpha/4^\ell} (c^{(\ell)}_r \alpha x)$.

However, such flexibility is yet too weak for generalizing Theorem \ref{thm:main-result-S-arbit-long-AP-fixed-common-diff} to the class of \emph{arbitrary} separated sets $\Lambda \subset \R$ with positive upper Beurling density.
Indeed, to adapt the proof technique of Theorem \ref{thm:main-result-S-arbit-long-AP-fixed-common-diff} to an arbitrary separated set $\Lambda \subset \R$,
one needs to extract from $\Lambda$ an arithmetic progression (resp.~an approximate arithmetic progression in the sense of \eqref{eqn:Lemma5-1-in-OU08-condition} in Lemma \ref{lem:Lemma5-1-in-OU08}) of length $2 \widetilde{M} {+} 1$ with common difference $P \alpha$ for some $P \in \{ c^{(\ell)}_r : r = 1, 2, 3, \ldots, 2^\ell \}$, where $\widetilde{M} = \widetilde{M}(\eta/4^\ell) \in \N$ is a large number chosen similarly as in \eqref{eqn:a-epsilonL-truncated-part}.
However, setting $\alpha = 1$ for simplicity, we note that Szemer\'{e}di's theorem (resp.~Lemma \ref{lem:Lemma5-1-in-OU08}) only guarantees the existence of an arithmetic progression (resp.~an approximate arithmetic progression) of length $2 \widetilde{M} {+} 1$
in $\Lambda$, where the common difference $P \in \N$ of the progression can be \emph{arbitrarily large}.
While the flexibility in choosing the set $\{ c^{(\ell)}_r : r = 1, 2, 3, \ldots, 2^\ell \}$ is certainly advantageous, there is no guarantee that the parameter $P$ will be in this set.
Hence, even for the improved set $S \subset [-\frac{1}{2}, \frac{1}{2}]$ given by \eqref{eqn:def-Veta-Vetaell-Z-2Z-3Z-improvement}, the general case of arbitrary separated sets $\Lambda \subset \R$ is still out of reach.

Note that the issue of $P \in \N$ being potentially very large is easily avoided when $\Lambda$ is assumed to have arbitrarily long arithmetic progressions with a fixed common difference $P$ (with $\alpha = 1$ chosen for simplicity), which has led to our first main result Theorem \ref{thm:main-result-S-arbit-long-AP-fixed-common-diff}.
\end{remark}

\appendix

\renewcommand{\thetheorem}{A.\arabic{theorem}}
\setcounter{theorem}{0}


\section{Related notions in Paley-Wiener spaces}
\label{sec:PW-language}


The Fourier transform\footnote{This is a nonstandard but equivalent definition of the Fourier transform which has no negative sign in the exponent.
This definition is employed only to justify the relation \eqref{eqn:relation-between-PW-and-L2}.
Alternatively, as in \cite{OU06,OU08} one could use the standard definition of the Fourier transform which has negative sign in the exponent, and define the Paley-Wiener space $PW(S)$ to be the space of Fourier transforms of $L^2(S)$.} is defined densely on $L^2(\R^d)$ by
\[
\mathcal{F} (f) := \widehat{f} (\omega) = \int f(x) \, e^{2 \pi i x \cdot \omega} \, dx
\quad \text{for} \;\; f \in L^1(\R^d) \cap L^2(\R^d) .
\]
It is easily seen that $\mathcal{F} : L^2(\R^d) \rightarrow L^2(\R^d)$ is a unitary operator satisfying $\mathcal{F}^2 = \mathcal{I}$, where $\mathcal{I} : L^2(\R^d) \rightarrow L^2(\R^d)$ is the reflection operator defined by $\mathcal{I} f (x) = f(-x)$, and thus $\mathcal{F}^4 = \mathrm{Id}_{L^2(\R^d)}$.
The \emph{Paley-Wiener space} over a measurable set $S \subset \R^d$ is defined by
\[
PW(S) := \{ f \in L^2(\R^d) : \supp \widehat{f} \subset S \}
\;=\; \mathcal{F}^{-1} \big[ L^2(S) \big]
\]
equipped with the norm $\| f \|_{PW(S)} := \| f \|_{L^2(\R^d)} = \| \widehat{f} \|_{L^2(S)}$, where $L^2(S)$ is embedded into $L^2(\R^d)$ by the trivial extension.
Denoting the Fourier transform of $f \in PW(S)$ by $F \in L^2(S)$, we see that
for almost all $x \in \R^d$,
\begin{equation}\label{eqn:relation-between-PW-and-L2}
f(x) = \big( \mathcal{F}^{-1} F \big) (x) = \int_S F(\omega) \, e^{- 2 \pi i x \cdot \omega} \, d\omega
= \big\langle F, e^{2 \pi i x \cdot (\cdot)} \big\rangle_{L^2(S)} .
\end{equation}
Moreover, if the set $S \subset \R^d$ has finite measure, then $f$ is continuous and thus \eqref{eqn:relation-between-PW-and-L2} holds for all $x \in \R^d$.

\begin{definition}\rm
Let $S \subset \R^d$ be a measurable set.
A discrete set $\Lambda \subset \R^d$ is called
\begin{itemize}
\item
\emph{a uniqueness set (a set of uniqueness)} for $PW(S)$ if the only function $f \in  PW(S)$ satisfying $f(\lambda) = 0$ for all $\lambda \in \Lambda$ is the trivial function $f=0$;

\item
\emph{a sampling set (a set of sampling)} for $PW(S)$
if there are constants $0 < A \leq B < \infty$ such that
\[
A \, \| f \|_{PW(S)}^2
\;\leq\;
\sum_{\lambda \in \Lambda} \big| f(\lambda) \big|^2
\;\leq\;
B \, \| f \|_{PW(S)}^2
\quad \text{for all} \;\; f \in PW(S) ;
\]

\item
\emph{an interpolating set (a set of interpolation)} for $PW(S)$ if for each $c = \{ c_ \lambda \}_{\lambda \in \Lambda} \in \ell_2(\Lambda)$ there exists a function $f \in  PW(S)$ satisfying $f(\lambda) = c_\lambda$ for all $\lambda \in \Lambda$.
\end{itemize}
\end{definition}

\medskip

It follows immediately from \eqref{eqn:relation-between-PW-and-L2} that
\begin{itemize}
\item
$\Lambda$ is a uniqueness set for $PW(S)$ if and only if $E(\Lambda)$ is complete in $L^2(S)$;

\item
$\Lambda$ is a sampling set for $PW(S)$ if and only if $E(\Lambda)$ is a frame for $L^2(S)$.
\end{itemize}
Also, we have the following characterization of interpolation sets for $PW(S)$ (see \cite[p.129, Theorem 3]{Yo01}):
\begin{itemize}
\item
$\Lambda$ is an interpolating set for $PW(S)$ if and only if there is a constant $A > 0$ such that
\begin{equation*}
A \, \| c \|_{\ell_2}^2
\;\leq\;
\Big\| \sum_{n\in\Z} c_ \lambda \, e^{2 \pi i \lambda (\cdot)} \Big\|_{L^2(S)}^2
\quad \text{for all} \;\; c = \{ c_ \lambda \}_{\lambda \in \Lambda} \in \ell_2(\Lambda) ,
\end{equation*}
meaning that the lower Riesz inequality of $E(\Lambda)$ for $L^2(S)$ holds.
\end{itemize}
Combining with the Bessel inequality (which corresponds to the upper Riesz inequality), we obtain a more convenient statement:
\begin{itemize}
\item
If $E(\Lambda)$ is a Bessel sequence in $L^2(S)$, then $\Lambda$ is an interpolating set for $PW(S)$ if and only if $E(\Lambda)$ is a Riesz sequence in $L^2(S)$.
\end{itemize}
In fact, this statement can be proved by elementary functional analytic arguments.
Indeed, if $E(\Lambda)$ is Bessel, i.e., if the synthesis operator $T : \ell_2 (\Lambda) \rightarrow L^2(S)$ defined by $T (\{ c_ \lambda \}_{\lambda \in \Lambda}) = \sum_{\lambda \in \Lambda} c_\lambda \, e^{2 \pi i \lambda (\cdot)}$
is a bounded linear operator
(equivalently, the analysis operator $T^* : L^2(S) \rightarrow \ell_2 (\Lambda)$ defined by $T^* F = \{ \langle F, e^{2 \pi i \lambda (\cdot)} \rangle_{L^2(S)} \}_{\lambda \in \Lambda}$ is a bounded linear operator),
then $T$ is bounded below (that is, the lower Riesz inequality holds) if and only if $T$ is injective and has closed range, if and only if $T^*$ has dense and closed range, i.e., $T^*$ is surjective, which means that $E(\Lambda)$ is an interpolating set for $PW(S)$ by \eqref{eqn:relation-between-PW-and-L2}.

The statement above is often useful because $E(\Lambda)$ is necessarily a Bessel sequence in $L^2(S)$ whenever $\Lambda \subset \R^d$ is separated and $S \subset \R^d$ is bounded \cite[p.135, Theorem 4]{Yo01}. Note that $\Lambda \subset \R^d$ is necessarily separated if $E(\Lambda)$ is a Riesz sequence in $L^2(S)$ (see Proposition \ref{prop:FR-rel-sep-RS-sep}).

\section{Proof of some auxiliary results}
\label{sec:proof-of-preliminaries}

\begin{proof}[\textbf{Proof of Lemma \ref{lem:RB-basic-operations}}]
To prove (a), note that for any $a \in \R^d$,
\[
T_{-a} [ E(\Lambda) ] = \{ e^{2\pi i \lambda (\cdot+a)} : \lambda \in \Lambda \}
= \{ e^{2\pi i \lambda a} \, e^{2\pi i \lambda (\cdot)} : \lambda \in \Lambda \} .
\]
Since the phase factor $e^{2\pi i \lambda a} \in \C$ for $\lambda \in \Lambda$ does not affect the Riesz basis property and Riesz bounds, it follows that $T_{-a} [ E(\Lambda) ]$ is a Riesz basis for $L^2(S)$ with optimal bounds $A$ and $B$.
Consequently, $E(\Lambda)$ is a Riesz basis for $L^2(S + a)$ with bounds $A$ and $B$. \\
For (b) and (c), note that the modulation $F(x) \mapsto e^{2 \pi i b x} F(x)$ is a unitary operator on $L^2(S)$ and that the dilation $F(x) \mapsto \sqrt{\sigma} F(\sigma x)$ is also a unitary operator from $L^2(S)$ onto $L^2(\frac{1}{\sigma} S)$.
It is easily seen from Proposition \ref{prop:facts-RieszBasesFR}(c) that
if $U : \mathcal{H}_1 \rightarrow \mathcal{H}_2$ is a unitary operators between two Hilbert spaces $\mathcal{H}_1$ and $\mathcal{H}_2$, and if $\{ f_n \}_{n \in \Z}$ is a Riesz basis for $\mathcal{H}_1$, then $\{ U f_n \}_{n \in \Z}$ is a Riesz basis for $\mathcal{H}_2$.
The parts (b) and (c) follow immediately from this statement.
\end{proof}

\begin{proof}[\textbf{Proof of Proposition \ref{prop:FR-rel-sep-RS-sep}}]
For simplicity, we will only consider the case $d=1$.

\noindent
(i) Assume that $D^+(\Lambda) = \infty$. This means that there is a real-valued sequence $1 \leq r_1 < r_2 < \cdots \rightarrow \infty$ such that
\[
\frac{\sup_{x \in \R} |\Lambda \cap [x,x{+}r_n] |}{r_n} > n
\quad \text{for all} \;\; n \in \N .
\]
Then for each $n \in \N$ there exists some $x_n \in \R$ satisfying
\[
\frac{|\Lambda \cap [x_n,x_n{+}r_n] |}{r_n} \geq n .
\]
For each $k \in \N$, we partition the interval $[x_n,x_n{+}r_n]$ into $k$ subintervals of equal length $\frac{r_n}{k}$, namely the intervals $\big[x_n,x_n{+}\frac{r_n}{k}\big], \, \ldots, \, \big[x_n{+}\frac{(k-1)r_n}{k},x_n{+}r_n\big]$.
Then at least one of the subintervals, which we denote by $I_{n,k}$, must satisfy
\begin{equation}
\label{eqn:Lambda-intersect-small-interval}
\frac{|\Lambda \cap I_{n,k} |}{| I_{n,k} |} \geq n  ,
\end{equation}
where $| I_{n,k} | = \frac{r_n}{k}$.
Letting $k \rightarrow \infty$, we see that
\[
\limsup_{r \rightarrow 0} \frac{\sup_{x \in \R} |\Lambda \cap [x,x{+}r] |}{r} = \infty .
\]
Define the function $g : \R \rightarrow \C$ by $g(x) = \frac{1}{\sqrt{|S|}} \chi_{S} (x)$ for $x \in \R$.
Then $\| g \|_{L^2(\R)} = \| g \|_{L^2(S)} = 1$ and $\widehat{g} (0) = \int_S g(x) \, dx = \sqrt{|S|}$.
Since $g \in L^1(\R)$, its Fourier transform $\widehat{g}$ is continuous on $\R$ and therefore exists $0< \delta < \frac{1}{2}$
such that $|\widehat{g}(\omega)| \geq \sqrt{|S|} / 2$ for all $\omega \in [-\frac{\delta}{2},\frac{\delta}{2}]$.
For each $n \in \N$,
we set $k_n := \lceil \frac{r_n}{\delta} \rceil \geq 2$, so that $k_n -1 < \frac{r_n}{\delta} \leq k_n$
and thus $\frac{\delta}{2} < \frac{r_n}{2(k_n -1)} \leq \frac{r_n}{k_n} \leq \delta$.
It then follows from \eqref{eqn:Lambda-intersect-small-interval} that
\[
|\Lambda \cap I_{n,k_n} |
\;\geq\; n \cdot |I_{n,k_n}|
\;\geq\; n \cdot \tfrac{r_n}{k_n}
\;>\; n \cdot \tfrac{\delta}{2} \; .
\]
For each $n \in \N$, we denote the center of the interval $I_{n,k_n}$ by $c_n \in \R$ and let $f_n \in L^2(S)$ be defined by $f_n(x) := e^{2\pi i c_n x} \, g (x)$ for $x \in S$. Then
\begin{equation}\label{eqn:Lambda-intersect-small-interval-lowerbound-final}
\begin{split}
\sum_{\lambda \in \Lambda}
\big| \langle f_n , e^{2\pi i \lambda (\cdot)} \rangle_{L^2(S)} \big|^2
&\;\geq\; \sum_{\lambda \in \Lambda \cap I_{n,k_n}}
\big| \langle g , e^{2\pi i (\lambda - c_n) (\cdot)} \rangle_{L^2(S)} \big|^2 \\
&\;=\; \sum_{\lambda \in \Lambda \cap I_{n,k_n}} \big| \widehat{g} (c_n - \lambda) \big|^2
\;>\; \big( n \cdot \tfrac{\delta}{2} \big) \cdot \tfrac{|S|}{4} ,
\end{split}
\end{equation}
where used that $c_n {-} \lambda \in [-\frac{\delta}{2},\frac{\delta}{2}]$ for all $\lambda \in \Lambda \cap I_{n,k_n}$, since $I_{n,k_n}$ is an interval of length $\frac{r_n}{k_n} \leq \delta$.
While $\| f_n \|_{L^2(S)} = \| g \|_{L^2(S)} = 1$ for all $n \in \N$, the right hand side of (\ref{eqn:Lambda-intersect-small-interval-lowerbound-final}) tends to infinity as $n \rightarrow \infty$. Hence, we conclude that $E(\Lambda)$ is not a Bessel sequence in $L^2(S)$ if $D^+(\Lambda) = \infty$. \\
(ii) Suppose to the contrary that $E(\Lambda)$ is a Riesz sequence in $L^2(S)$ with Riesz bounds $A$ and $B$, but the set $\Lambda \subset \R$ is not separated.
Then there are two sequences $\{ \lambda_n \}_{n=1}^\infty$ and $\{ \lambda_n' \}_{n=1}^\infty$ in $\Lambda$ such that $| \lambda_n -  \lambda_n' | \rightarrow 0$ as $n \rightarrow \infty$.
Note that $S \subset \R$ is a finite measure set, and for each $x \in S$ we have $| e^{2\pi i \lambda_n x} - e^{2\pi i \lambda_n' x} | \leq 2$ and $e^{2\pi i \lambda_n x} - e^{2\pi i \lambda_n' x} \rightarrow 0$ as $n \rightarrow \infty$.
Thus, we have $\lim_{n \rightarrow \infty} \int_S \big| e^{2\pi i \lambda_n x} - e^{2\pi i \lambda_n ' x} \big| ^2 \, dx = 0$ by the dominated convergence theorem.
For $\lambda \in \Lambda$, let $\delta_{\lambda} \in \ell_2 (\Lambda)$ be the Kronecker delta sequence supported at $\lambda$, that is,
$\delta_{\lambda} (\lambda') = 1$ if $\lambda' = \lambda$, and $0$ otherwise.
Then since $E(\Lambda)$ is a Riesz sequence in $L^2(S)$, we have
\[
\begin{split}
2
&\;=\; \| \delta_{\lambda_n} - \delta_{\lambda_n'} \|_{\ell_2(\Lambda)}^2 \\
&\;\leq\;
\tfrac{1}{A} \,
\big\| e^{2\pi i \lambda_n (\cdot)} - e^{2\pi i \lambda_n ' (\cdot)} \big\|_{L^2(S)}^2
=
\tfrac{1}{A} \, \int_S \big| e^{2\pi i \lambda_n x} - e^{2\pi i \lambda_n ' x} \big| ^2 \, dx
\;\; \rightarrow \;\; 0 ,
\end{split}
\]
yielding a contradiction.
\end{proof}

\begin{proof}[\textbf{Proof of Lemma \ref{lem:Lemma5-1-in-OU08}}]
Let $\Lambda = \{ \lambda_n \}_{n \in \Z}$ with $\lambda_n < \lambda_{n+1}$ for all $n$,
and fix any $\delta > 0$. Choose a sufficiently large number $N \in \N$ so that $\frac{1}{N} < \tau := \min \{ \Delta(\Lambda) , 2 \delta \}$, where $\Delta(\Lambda) := \inf\{ |\lambda - \lambda'| : \lambda \neq \lambda' \in \Lambda \}$ is the separation constant of $\Lambda$ (see Section~\ref{subsec:prelim-density}).
Consider the perturbation $\widetilde{\Lambda} \subset \frac{1}{N} \Z$ of $\Lambda$, obtained by rounding each element of $\Lambda$ to the nearest point in $\frac{1}{N} \Z$ (if $\lambda \in \Lambda$ is exactly the midpoint of $\frac{k}{N}$ and $\frac{k+1}{N}$, then we choose $\frac{k}{N}$).
Since $\Delta(\Lambda) > \frac{1}{N}$, all elements in $\Lambda$ are rounded to distinct points in $\frac{1}{N} \Z$, i.e., the set $\widetilde{\Lambda} = \{ \widetilde{\lambda}_n \}_{n \in \Z} \subset \frac{1}{N} \Z$ has no repeated elements.
Clearly, there is a 1:1 correspondence between $\lambda_n$ and $\widetilde{\lambda}_n$, and we have $| \lambda_n - \widetilde{\lambda}_n | \leq \frac{1}{2N} < \frac{\tau}{2} \leq \delta$ for all $n \in \Z$.

We claim that for any $M \in \N$ there exist constants $c \in \N$, $d \in \frac{1}{N} \Z$, and an increasing sequence $\widetilde{s}(-M) < \widetilde{s}(-M{+}1) < \ldots < \widetilde{s}(M)$ in $\widetilde{\Lambda} \subset \frac{1}{N} \Z$ satisfying
\[
\widetilde{s}(j) = c j + d
\quad \text{for} \;\; j = -M , \ldots, M .
\]
Once this claim is proved, it follows that the sequence $\{ s(j) \}_{j=-M}^M \subset \Lambda$ corresponding to $\{ \widetilde{s}(j) \}_{j=-M}^M \subset \widetilde{\Lambda}$, satisfies the condition \eqref{eqn:Lemma5-1-in-OU08-condition} as desired.

To prove the claim,
consider the partition of $N \widetilde{\Lambda} \; (\subset \Z)$ based on residue modulo $N$, that is, consider the sets $N\widetilde{\Lambda} \cap N\Z, N\widetilde{\Lambda} \cap (N\Z{+}1), \ldots, N\widetilde{\Lambda} \cap (N\Z{+}N{-}1)$.
Since $D^+ (\widetilde{\Lambda}) = D^+ (\Lambda) > 0$, at least one of these $N$ sets must have positive upper density, i.e., $D^+ (N\widetilde{\Lambda} \cap (N\Z{+}u) ) > 0$ for some $u \in \{ 0 , \ldots, N-1 \}$.
Then Szemer\'{e}di's theorem implies that for any $M \in \N$ the set $N\widetilde{\Lambda} \cap (N\Z{+}u)$ contains an arithmetic progression of length $2M{+}1$, that is, $\{ c_0 j + d_0 : j = -M , \ldots, M \} \subset N\widetilde{\Lambda} \cap (N\Z{+}u)$ for some $c_0 \in \N$ and $d_0 \in \Z$.
This means that there is an increasing sequence $\widetilde{s}(-M) < \widetilde{s}(-M{+}1) < \ldots < \widetilde{s}(M)$ in $\widetilde{\Lambda}$ satisfying
\[
N \, \widetilde{s}(j) = c_0 j + d_0
\quad \text{for} \;\; j = -M , \ldots, M .
\]
Since the numbers $c_0 j + d_0$, $j = -M , \ldots, M$ are in $N\Z{+}u$, it is clear that
$c_0 \in N \N$ and $d_0 \in N\Z{+}u$.
Thus, setting $c := \frac{1}{N} c_0 \in \N$ and $d := \frac{1}{N} d_0 \in \Z{+}\frac{u}{N} \subset \frac{1}{N} \Z$, we have $\widetilde{s}(j) = cj + d$ for $j = -M, \ldots, M$, as claimed.

Finally, one can easily force the constant $c \in \N$ to be a multiple of any prescribed number $L \in \N$.
This is achieved by considering the partition of $N \widetilde{\Lambda} \; (\subset \Z)$ based on residue modulo $LN$, instead of modulo $N$.
\end{proof}

\section*{Acknowledgments}

The author acknowledges support by the DFG Grants PF 450/6-1 and PF 450/9-1.





\end{document}